\documentclass[sn-chicago,Numbered]{sn-jnl}

\usepackage{graphicx}%
\usepackage{multirow}%
\usepackage{amsmath,amssymb,amsfonts,dsfont}%
\usepackage{amsthm}%
\usepackage{mathrsfs}%
\usepackage[title]{appendix}%
\usepackage{xcolor}%
\usepackage{textcomp}%
\usepackage{manyfoot}%
\usepackage{booktabs}%
\usepackage{listings}%
\usepackage{tikz-cd}%
\usepackage{tikz}

\usepackage{geometry}
\newcommand{\cP}{{\mathcal P}}

\newcommand{\cE}{{\mathcal E}}
\newcommand{\cO}{{\mathcal O}}
\newcommand{\cU}{{\mathcal U}}

\newcommand{\RR}{{\mathbf{R}}}

\newcommand{\cLie}{\mathcal{L}ie}

\theoremstyle{thmstyleone}%
\newtheorem{theorem}{Theorem}
\newtheorem{lemma}[theorem]{Lemma}%
\newtheorem{corollary}[theorem]{Corollary}

\theoremstyle{thmstyletwo}%
\newtheorem{example}{Example}%

\theoremstyle{thmstylethree}%
\newtheorem{definition}{Definition}%

\numberwithin{theorem}{section}
\numberwithin{definition}{section}

\begin{document}

\title[$n\text{-}Lie_d$ operad and its Koszul Dual]{$n\text{-}Lie_d$ Operad and its Koszul Dual}


\author{\fnm{Cody} \sur{Tipton}}

\affil{\orgdiv{Mathematics Department}, \orgname{University of Washington, Seattle}}


\abstract{ We study the operad $n\text{-}Lie_d$, whose algebras are graded $n$-Lie algebra with degree $d$ $n$-arity operations, which were introduced in Nambu mechanics and later studied in the algebraic setting with Filippov.  We compute the Koszul dual of $n\text{-}Lie_d$, called $n\text{-}Com_{-d+n-2}$, whose relations are derived from the Specht module $S^{(n,n-1)}$ for a partition $(n,n-1)$ of $2n-1$.  The intrinsic connection between these two operads come from the eigenvalues of the sequence of graphs $\{\cO_n\}_{n\geq 2}$, called the Odd graphs, whose spectrum is related to the lower triangular sequence $\{\cE_{r,n}\}$, called the Catalan triangle. 
}




\maketitle
\tableofcontents

\section*{Introduction}
\indent An $n$-Lie algebra is a vector space $L$ with an $n$-arity antisymmetric bracket $[-,\dots, -]:L^{\otimes n}\rightarrow L$ such that for any $n-1$ elements, $[x_1,\dots, x_{n-1},-]$ is a derivation on $L$.  These are a natural generalization of Lie algebras by extending the Jacobian identity through the viewpoint as a derivation property.  The most natural example of a $n$-Lie algebra is taking the polynomial ring $A=k[x_1,\dots, x_n]$ and defining 
\begin{align*}
    [f_1,\dots, f_n] = \text{Jac}(f_1,\dots, f_n),
\end{align*}
where the right-hand side is the Jacobian of the $n$-polynomials $f_1,\dots, f_n\in A$.  The concept of $3$-Lie bracket was first introduced by Nambu in \cite{Nambu}, to generalize the ideas of Hamiltonian mechanics to more than one Hamiltonian.  There have been many more applications of $n$-Lie algebras in the physics literature as in superconformal field theory of multiple M2-branes with metric $n$-Lie algebras, for more physics applications see (cf. \cite{Nambu}, \cite{Physics2},\cite{Physics3},\cite{Physics4},\cite{Physics5},\cite{Physics6}).     
\\
\indent In the algebraic setting, the general concept of $n$-Lie algebras was first introduced by Filippov in \cite{filippov}, who studied the representation theory and classified $n$-Lie algebras of dimension $\leq (n+1)$ in characteristic $\neq 2$. Later, Kasymov \cite{Kasymov1987} started the pioneering work of studying the nilpotency and solvability of $n$-Lie algebras and prove a generalization of Engel's theorem and Cartan's solvability criterion for these types of algebras.  More recently, the use of $n$-Lie algebras have become a tool in the use of automorphism problems of the polynomial ring using the natural $n$-Lie Poisson structure from the Jacobian, see \cite{huang2023valuation}.
\\
\indent As with most algebraic structures there is a corresponding object, called an operad, which controls almost everything about the defining relations, deformation theory, Koszul duality, universal enveloping algebra, ect.. of its algebras.  These can be thought as spaces of rooted trees, representing the composition of $m$-arity operations on its class of algebras with the relations of the algebras embedded in these spaces.  In this paper we  study the operad $n\text{-}Lie_d$, whose algebras are graded $n$-Lie algebras with degree $d$ $n$ arity brackets.  Note that $2\text{-}Lie_0$ is just the operad $Lie$, the operad controlling Lie algebras.
\\
\indent One of the fundamental properties of Lie algebras is the Koszul dual of $Lie$ is the operad $Com$, the operad controlling commutative algebras.  This was first noticed by the works of Quillen and Sullivan \cite{QuillenRHT,Sullivan1977} in their study of the rational homotopy category and connecting them to the category of homotopy Lie algebras and commutative dg algebras, respectively.  We extend the Koszul duality between $Lie$ and $Com$, by finding the Koszul dual of $n\text{-}Lie_d$, which we call $n\text{-}Com_{-d+n-2}$, for all $n\geq 2$ and $d\in\mathbb{Z}$. More specifically, our main result is the following.
\\
\indent
\begin{theorem}
    \indent For all $n\geq 2$ and $d\in\mathbb{Z}$, the operads $n\text{-}Lie_d$ and $n\text{-}Com_{-d+n-2}$ are Koszul dual.
\end{theorem}
\indent
\\
\indent To accomplish this, we use the extension of the work of Ginzburg and Kapranov in their study of binary quadratic operads in \cite{ginzburg} to the case where our generators are concentrated in arity $n$, defined in \cite{}, with either the trivial or signed representation, which we call $n$-quadratic operads.  In the same way as Ginzburg and Kapranov, the Koszul dual of a $n$-quadratic operad $\cP(E,R)$, where $E$ is the generators and $R$ are the relations, to be the operad $\cP^{\text{!}}=\cP(E^{\vee},R^{\perp})$, where $E^{\vee}$ is a signed graded $k$-linear dual and $R^{\perp}$ is the orthogonal relations defined through some non-degenerate bilinear form.
\\
\indent The algebras over $n\text{-}Com_{d}$ are graded $k$-modules $C$ equipped with a symmetric $n$-arity operation $m_n:C^{\otimes n}\rightarrow C$ of degree $d$ satisfying the relation 
\begin{align*}
\sum_{\sigma\in C_{T_{\lambda}}}Sgn(\sigma)(m_n\circ_1m_n)^{\sigma}=0
\end{align*} based on $\lambda$-polytabloid on the partition $\lambda=(n,n-1)$ in the Specht module $S^{\lambda}$, and where $C_{T_{\lambda}}$ is the subgroup preserving the columns of the associated Young diagrams.  In particular, this will show $(n-Lie_d)(2n-1)\cong \uparrow^dS^{\lambda}$, which has dimension $C_n$, the $n$th Catalan number, which was first proved by Tamar Friedmann, Phil Hanlon, Richard P. Stanley, and Michelle L Wachs in \cite{FRIEDMANN2021107570}.
To show $n\text{-}Com_{-d+n-2}$ is the Koszul dual of $n\text{-}Lie_d$ , we connect the relations in $n\text{-}Com_{-d+n-2}$ to the orthogonal relations of $n\text{-}Lie_d$ through the eigenspaces of a sequence of graphs $\{\cO_n\}$, called the Odd graphs which are a certain subclass of Kesner graphs, whose eigenvalues $(-1)^{n+1}$ have multiplicity $C_{n}$.
\\
\indent As in the case of $n$-Lie algebras, it is very difficult to find examples of $n$-Com algebras of degree $d$ for general $n\geq 3$ as the relations are very complicated.  Our two main examples of $n$-Com algebras, for $n\geq 3$, are derived from looking at derivations on the polynomial ring $A=k[x_1,\dots, x_n]$.  We put a $n$-Com algebra on $A$ by using the multiplication 
\begin{align*}
    m_n^i(f_1,\dots, f_n) = \frac{\partial}{\partial x_i}(f_1\cdots f_n)
\end{align*}
for $n\geq 3$, $1\leq i\leq n$, and $f_1,\dots, f_n\in A$. On the other hand, we can put a $n$-Com algebra structure on the space $\text{Der}(A)$ of derivations on $A$ through the multiplication 
\begin{align*}
    b_n(D_1,\dots, D_n) = \sum_{i_1,\dots, i_n} \delta_{i_1,\dots, i_n}^j D_1(x_{i_1})\cdots D_n(x_{i_n})\frac{\partial}{\partial x_j},
\end{align*}
where $\delta_{i_1,\dots, i_n}^j=1$ if $j\in\{i_1,\dots, i_n\}$ and zero otherwise.

\section*{Acknowledgement}
\indent I would like to express my deepest gratitude to my advisor Dr. James Zhang for their expert guidance, encouragement, and for introducing me to $n$-Lie algebras, to Zihao Qi for reading over this paper, finding errors,  and help me correct them, to Albert Artilles for listening to me pester him about nonsensical algebra at the darkest time of the day, and to Kevin Liu for pointing me in the right direction to solve some of the combinatorics with respect to the Catalan numbers.
\section*{Notation and Conventions}
\indent Let $n\geq 2$ and denote by $\Sigma_n$ the symmetric group on $n$ letters.  If $m\geq n$, we will consider $\Sigma_n$ as a subgroup of $\Sigma_m$, whose elements fix the last $m-n$ elements of $\{1,\dots, m\}$, and we denote by $\overline{\Sigma}_n$ the subgroup of $\Sigma_m$ whose elements fix the first $m-n$ elements of $\{1,\dots, m\}$.  Denote by $Sh(n,m)$ the set of $(n,m)$-shuffles in $\Sigma_{n+m}$.
\\
\indent Let $k$ be a field and denote by $GVect_k$ the category of graded $k$-modules.  For most of the definitions, there is no need for a restriction on the characteristic of $k$, but it is needed in our final result in theorem \ref{maintheorem}.  For simplicity, we will assume that the characteristic of $k$ is always $0$, unless otherwise noted. 
\\
\indent Let $V$ be a graded $k$-module and define the right action of $\Sigma_n$ on $V^{\otimes n}$ as follows: for $v_1,\dots, v_n\in V$, define
\begin{align*}
    \sigma\cdot (v_1\otimes\cdots\otimes v_n) = \xi(\sigma,v_1,\dots, v_n) v_{\sigma^{-1}(1)}\otimes\cdots\otimes v_{\sigma^{-1}(n)}
\end{align*}
where $\xi(\sigma,v_1,\dots, v_n)$ is the Koszul sign rule from permutating the elements in the tensor product by $\sigma^{-1}$.   Furthermore, if we have a $n$-arity function $f:V^{\otimes n}\rightarrow V$ of degree $d$, we define 
\begin{align*}
    f^{\sigma}(v_1\otimes\cdots\otimes v_n) = f(\sigma\cdot v_1\otimes\cdots\otimes v_n) = \xi(\sigma,v_1,\dots, v_n)f(v_{\sigma^{-1}(1)}\otimes\cdots\otimes v_{\sigma^{-1}(n)}).
\end{align*}

\section{Preliminaries}

    \subsection{Rooted Trees}
        \indent There are various way to define the free operad in the setting of vector spaces, the route we will take is through the combinatorial description using rooted trees.  There is a natural relationship with identifying operations in an operad with its corresponding rooted tree which gives an intuitive description of the composition of the free operad in terms of grafting.  Here we will state some definitions and notation for rooted trees that will be helpful later in the paper.  For more information, see \cite{Donald}.

        \begin{definition}
            \indent Let $m,n\geq 0$.  A directed $(m,n)$-graph is a quadruple $G=(V,E,in_G,out_G)$ consisting of 
            \begin{itemize}
                \item a directed graph $(V,E)$, where $V$ are the abstract vertices and $E$ are the ordered edges, and 
                \item disjoint subsets $in_G$ and $out_G$ of $V$ such that the following conditions hold.
                \begin{itemize}
                    \item $|in_G|=m$ and $|out_G|=n$.
                    \item Each $v\in in_G$, we have $|in(v)|=0$ and $|out(v)|=1$, where $in(v)$ is the set of input edges of $v$ and $out(v)$ is the set of outgoing edges of $v$.  
                \end{itemize}
            \end{itemize}
        \end{definition}
        \indent\\
        \indent For any directed $(m,n)$-graph $G$, define $Vt_G$ to be the set of elements in $V$ that are not in $in_G\cup out_G$.  We call elements of $in_G$ the inputs of $G$, the elements of $out_G$ the outputs of $G$,  and call elements of $Vt_G$ the vertices of $G$.  

        \begin{definition}
            \indent A rooted $m$-tree $T$ is a connected,acyclic, directed $(m,1)$-graph $(V,E,in_T,out_T)$ such that $|out(v)|=1$ for every $v\in Vt_T$.  
        \end{definition}
        \indent\\
        \indent For an $m$-rooted tree $T$, we call the unique outgoing edge of $T$, the root edge, and denote by $r_t$, the unique initial vertex $v$ of the root edge, called the root vertex, provided that $v$ is not in $in_T$, otherwise it does not exist.

        \begin{definition}
            \indent An isomorphism $F:(T,V,E,in_T),r)\rightarrow (T',V',E',in_{T'}, r')$ of $n$-trees is a bijection $F_V:V\rightarrow V'$ of the vertices such that $F$ is the identity on $in_T=[n]=in_{T'}$ and $F(r_T)=r_{T'}$, if they exist, and a bijection $F_E:E\rightarrow E'$ that preserves the ordering i.e. $e=(x,y)\in E$ if and only if $F_E(e)=(F_V(x), F_V(y))\in E'$. Denote by $Tree(n)$ the collection of isomorphism classes of trees with $n$-leaves. 
        \end{definition}
        \indent
        
            \indent Suppose $T$ is an $m$-tree, and $\sigma\in\Sigma_m$.  We define $\sigma^*(T)$ to be the tree with the same vertices, but with the inputs $in_T$ re-indexed by $\sigma$ through the induced map $\sigma:in_T=[n]\rightarrow in_{\sigma^*(T)}=[n]$.  
                \indent This induces natural isomorphisms $\sigma^*:E_T\rightarrow E_{\sigma^*(T)}$ and $\sigma^*:in_T(v)\rightarrow in_{\sigma^*(T)}(v)$ for each $v\in V_T$ by applying $\sigma$ to elements of $in_T$ and the identity on the rest.     

            \indent Here we will briefly describe the process of grafting two trees, as this will be important for defining the composition of the free operad.
    
                \begin{definition}
                    \indent Suppose $(T_1,V_1,E_1,In_{T_1},r_1)$ is a $m$-rooted tree and $(T_2,V_2,E_2,In_{T_2}, r_2)$ is a $n$-rooted trees.  Let $e$ be the $i$th edge of $T_1$, i.e. the unique outgoing edge pertaining to the abstract vertex $i$.  We define $S=T_1\circ_eT_2$ to be the $n+m-1$ rooted tree with the following properties:
                    \begin{itemize}
                        \item The set of abstract vertices $V= V_1\coprod V_2$;
                        \item The set of edges $E= E_1\coprod E_2/ (e\sim e_{r_{T_2}})$\\
                        \item The root vertex $r_S=r_{T_1}$;
                        \item The inputs $In_S =[n+m-1]$ 
                    \end{itemize}
                \end{definition}
            \indent Denote by $\downarrow$ for the trivial $1$-tree with no vertices, which is the unique tree up to $1$-tree isomorphism.  For $m\geq 1$, denote by $Cor_m$ the $m$-rooted tree with a single internal vertex $v$ and edges $(i,v)$ for $1\leq i\leq m$, these trees are called the $m$-Corolla rooted trees and are very important, since all other trees are grafting of such trees. For an example of grafting two corolla trees, see \ref{grafting}. Another important property of Corolla trees, it that they are invariant under any permutation of $\Sigma_n$: if $\sigma\in\Sigma_m$, then $\sigma^*(Cor_m)\cong Cor_m$ through an identity map.

            \begin{figure}[h!]
                \centering

\tikzset{every picture/.style={line width=0.75pt}} 

\begin{tikzpicture}[x=0.75pt,y=0.75pt,yscale=-1,xscale=1]

\draw    (70,70) -- (110,110) ;
\draw    (110,110) -- (110,160) ;
\draw    (150,70) -- (110,110) ;
\draw    (110,70) -- (110,110) ;
\draw    (190,70) -- (230,110) ;
\draw    (230,110) -- (230,160) ;
\draw    (270,70) -- (230,110) ;
\draw    (340,110) -- (365,132.22) ;
\draw    (365,132.22) -- (365,160) ;
\draw    (390,110) -- (365,132.22) ;
\draw    (365,110) -- (365,132.22) ;
\draw    (365,57.78) -- (390,80) ;
\draw    (390,80) -- (390,110) ;
\draw    (415,57.78) -- (390,80) ;

\draw (159,112.4) node [anchor=north west][inner sep=0.75pt]    {$\circ _{e}$};
\draw (285,112.4) node [anchor=north west][inner sep=0.75pt]    {$=$};
\draw (67,52.4) node [anchor=north west][inner sep=0.75pt]    {$1$};
\draw (107,52.4) node [anchor=north west][inner sep=0.75pt]    {$2$};
\draw (147,52.4) node [anchor=north west][inner sep=0.75pt]    {$3$};
\draw (187,52.4) node [anchor=north west][inner sep=0.75pt]    {$1$};
\draw (267,52.4) node [anchor=north west][inner sep=0.75pt]    {$2$};
\draw (337,92.4) node [anchor=north west][inner sep=0.75pt]    {$1$};
\draw (361,92.4) node [anchor=north west][inner sep=0.75pt]    {$2$};
\draw (361,42.4) node [anchor=north west][inner sep=0.75pt]    {$3$};
\draw (407,42.4) node [anchor=north west][inner sep=0.75pt]    {$4$};

\end{tikzpicture}
        \label{grafting}
        \caption{Grafting of $Cor_3$ and $Cor_2$ at $e$, where $e$ is the $3$rd input edge of $Cor_3$.}
            \end{figure}

    \subsection{Operads}\label{operads}

    \indent Peter May introduced operads in \cite{may} to study the loop spaces in algebraic topology.  These objects parameterize classes of different types of algebras by essentially describing the $n$-ary operations and their relationships.  We will be following the notation and definitions in \cite{loday} for this section.
    \\
    \indent The underlying structure for operads are $\Sigma$-modules, which are a family $M=(M(1),\dots, M(n),\dots)$ of right $k[\Sigma_n]$-modules for $n\geq 1$.  This is equivalent to a functor $M:\Sigma^{op}\rightarrow Vect_k$, where $\Sigma$ is the permutation groupoid.  One should think of each $M(n)$ as holding the $n$-ary operations in the form of $n$-arity trees, which is the intuition for the free operad later discussed in section \ref{free operad}.    
    \\
    \indent A morphism $f:M\rightarrow N$ of $\Sigma$-modules is a natural transformation between their functors, i.e. a collection of right $k[\Sigma_n]$-module homomorphisms $f_n:M(n)\rightarrow N(n)$.
    \\
    \indent To help define the combinatorial free operad on a $\Sigma$-module,  we will need the use of linear species, which are similar to $\Sigma$-modules, but are defined on more general finite sets.  Any $\Sigma$-module $M$ can be made into a linear species $\widetilde{M},$ where for any finite set $X$ with cardinality $n$, we define
    \begin{align*}
                \widetilde{M}(X) = \left( \bigoplus_{f\in Bij([n],X)} M(n)\right)_{\Sigma_n},
            \end{align*}
            where $Bij$ is the category of finite sets with their bijections.
             Equivalence classes in $\widetilde{M}$ are represented by $(f;\mu)$, where $f\in Bij([n],X)$ and $\mu\in M(n)$, and the right $\Sigma_n$ action inside of the coinvariants is defined as
            \begin{align*}
                (f;\mu)^{\sigma}= (f\sigma; \mu) 
            \end{align*}
            If $h:X\rightarrow Y$ is any bijection of sets with cardinality $n$, then we obtain a isomorphism $\widetilde{M}(h):\widetilde{M}(Y)\rightarrow \widetilde{M}(X)$ where
            $$\widetilde{M}(h)(f;\mu) = (h^{-1}f; \mu)$$
            \\
            \indent If we let $\iota:\Sigma\rightarrow Bij$ be the natural inclusion functor, then we obtain a natural isomorphism of functors $\tau:\widetilde{M}\iota\rightarrow M$ with $\tau_n([f;\mu]) = \mu^{f^{-1}}$.  In other words, if $\sigma\in \Sigma_n$, then we have
            \begin{align*}
                M(\sigma)\tau_n([f;\mu]) &=M(\sigma)(\mu^{f^{-1}}) = (\mu^{f^{-1}})^{\sigma}= \mu^{f^{-1}\sigma}\\
                &= \tau_n([\sigma^{-1}f;\mu]) = \tau_n(\widetilde{M}(\sigma)([f;\mu])). 
            \end{align*}
            This implies that $\widetilde{M}([n])\cong M(n)$ for any $n\geq 1$.  
            \indent\\
            \indent There are a few different equivalent ways to define a operad in the literature.  For our purposes, we will use the classical definition as described in \cite{may}, as it is more useful for computations with the free operad in terms of rooted trees.  For more information about operads and their different equivalent ways to define them in the algebraic setting see \cite{loday}.
        \begin{definition}
            A pseudo-operad is a tuple $(\mathcal{O}, \circ, \eta)$ consisting of the following data;
            \begin{itemize}
                \item $\mathcal{O}$ is a $\Sigma$-module;
                \item for each $n,m\geq 1$ and $1\leq i\leq n$ it is equipped with a $k$-linear map
                    $$ -\circ_i-: \mathcal{O}(n)\otimes \mathcal{O}(m)\rightarrow \mathcal{O}(n+m-1)$$
                \item and a $k$-linear map $\eta:k\rightarrow \mathcal{O}(1)$
            \end{itemize}
            satisfying the following axioms.
                \begin{itemize}
                    \item For $n\geq 2$ and $1\leq i<j\leq n$, then the horizontal associativity for $\mu\in \mathcal{O}(n)$, $\nu\in \mathcal{O}(m)$ and $\gamma\in \mathcal{O}(l)$ is
                    \begin{align*}
                        (\mu\circ_i \gamma) \circ_{j-1+l} \nu = (\mu\circ_j \nu)\circ_i \gamma.
                    \end{align*}
                    \item For $n,m\geq 1$, $1\leq i\leq n$ and $1\leq j\leq m$, then the vertical associativity relation for $\mu\in\cO(n)$, $\nu\in \cO(m)$, and $\gamma\in \cO(l)$ is
                    \begin{align*}
                        \mu\circ_i (\nu\circ_j\gamma) = (\mu\circ_{i} \nu)\circ_{i-1+j}\gamma.
                    \end{align*}
                    \item For $n\geq 1$ and $1\leq i\leq n$, and any $\mu\in \cO(n)$ we have
                    \begin{align*}
                        \mu\circ_i\eta(1) = \mu\\
                    \end{align*}
                    and
                    \begin{align*}
                        \eta(1)\circ_1\mu=\mu.
                    \end{align*}
                    \item For $n\geq 1$, $1\leq i\leq n$, $\sigma\in \Sigma_n$, and $\tau\in \Sigma_m$, the equivariance relation for $\mu\in\cO(n)$ and $\nu\in \cO(m)$ is
                    \begin{align*}
                        (\mu\circ_{\sigma(i)} \nu)^{\sigma\circ_i\tau} = \mu^{\sigma}\circ_i\nu^{\tau}
                    \end{align*}
                \end{itemize}
                where $\sigma\circ_i\tau\in\Sigma_{n+m-1}$ is the permutation in 
                \begin{align*}
                    (\sigma\circ_i\tau) = \sigma(id\oplus\cdots id\oplus \tau\oplus id\oplus\cdots \oplus id)
                \end{align*}
                with $i-1$ identity maps on the left of $\tau$ and $n-i$ identity maps on the right of $\tau$.  
        \end{definition}
        \indent
        \\
        \indent
        \begin{example}\label{ex: operads}
            \indent The following are a list of the traditional examples of algebraic operads in the literature.
            \begin{itemize}
                \item \indent Given any vector space $V$, define $\text{End}_V$ to be the $\Sigma$-module such that $\text{End}_V(n)=\text{Hom}_k(V^{\otimes n},V)$  with the natural right $\Sigma_n$ action induced from left action on $V^{\otimes n}$.  For $f\in \text{End}_V(n)$ and $g\in \text{End}_V(m)$, define the partial composition
                \begin{align*}
                    (f\circ_ig)(v_1,v_2,\dots, v_{n+m-1}) = f(v_1,\dots, v_{i-1}, g(v_i,\dots, v_{i+m-1}),v_{i+m},\dots, v_{n+m-1})
                \end{align*}
                and define the unit map $\eta:k\rightarrow \text{End}_V(1)=\text{Hom}_k(V,V)$ as sending $1$ to the identity map.  This becomes an operad in a very natural way.
                \item The associative operad has $\Sigma$-module $Ass$, where $Ass(n)=k[\Sigma_n]$ for $n\geq 1$.  Let $\mu_n$ be the generator of $k[\Sigma_n]$ as a right $k[\Sigma_n]$-module and we define the composition $\mu_n\circ_i\mu_m = \mu_{n+m-1}$.  This becomes a operad in a natural way and its algebras are exactly the non-unital associative $k$-algebras.
                \item The commutative operad has $\Sigma$-module $Com$ with $Com(n)=k\nu_n$, where $\nu_n$ has the trivial action.  Defining the partial composition as in $Ass$, this becomes an algebraic operad as well.
                \item The Lie operad has $\Sigma$-module $Lie$ where $Lie(n)$ is the vector space of rooted planar $n$-trees with the inputs labeled by $\{1,\dots, n\}$ quotient by the relations of antisymmetry and the Jacobian identity.  This becomes a algebraic operad through the grafting of trees.
            \end{itemize}
        \end{example}

     \subsection{Free Operad}\label{free operad}
        \indent As with any algebraic object, we would like to construct a free object so we can generate examples through generators and relations.  One way to think of the free operad associated to a $\Sigma$-module $E$ is as the collection of $n$ rooted trees with the vertices labeled by the elements of $E$.  To do this, we need to use the language of linear species as explained in \ref{operads} to help us define this notion.  There are various ways to define the free operad associated to a $\Sigma$-module, but for our purposes we will follow the construction outlined in \cite{Donald} and in \cite{benoit}.  For the other equivalent way to define the free operad in terms of the monoidal structure on the category of $\Sigma$-modules, see \cite{loday}.
        \begin{definition}
            Let $M$ be a linear species and suppose we have an isomorphism class $[T]\in Tree(n)$ with $m=|V_T|$.  Define the $M$-decoration of $[T]$ as 
            \begin{align*}
                M[T]= \left( \bigoplus_{f\in Bij([m], V_T)}M(in_T(f(1)))\otimes\cdots\otimes M(in_T(f(m)))\right)_{\Sigma_m}
            \end{align*}
            where elements are expressed as $(f;\mu_1,\dots, \mu_m)$ for $f\in Bij([m],V_T)$ and $\mu_i\in M(in_T(f(i)))$ and where the right action of $\sigma\in\Sigma_m$ on an element $(f;\mu_1,\dots, \mu_m)$ is defined as
            \begin{align*}
                (f;\mu_1,\dots, \mu_m)^{\sigma}=(f\sigma; \mu_{\sigma^{-1}(1)},\dots, \mu_{\sigma^{-1}(m)})
            \end{align*}
            If $T$ is the unique tree $\downarrow$, then we define $M[\downarrow]=k$.
        \end{definition}
        \indent Note that this definition is independent of the representative of the class $[T]$, since isomorphic trees induce identities on $M(In_T(v))$.  Furthermore, if $\sigma\in\Sigma_n$, then we have an isomorphism $\sigma^*:M[\sigma^*(T)]\rightarrow M[T]$ induced by the isomorphisms $M(\sigma^*):M(In_{\sigma^*(T)}(v))\rightarrow M(In_T(v))$ coming from $\sigma^*_v:In_T(v)\rightarrow In_{\sigma^*(T)}(v)$ for each $v\in V_T$.
        \\
        \begin{lemma}
            \indent Let $E$ be a $\Sigma$-module and let $T_1$ be an $n$-tree with $|V_{T_1}|=q$ and $T_2$ be a $m$-tree with $|V_{T_2}|=p$.  Suppose $e$ is the $i$th edge of $T_1$, then we have the map.
            \begin{align*}
                \Psi&:E[T_1]\otimes E[T_2]\rightarrow E[T_1\circ_eT_2]\\
               [f_1;\mu_1\otimes\cdots\otimes \mu_q]&\otimes [f_s;\nu_1\otimes\cdots\otimes \nu_p]\mapsto [f_1\times f_2; \mu_1\otimes\cdots\otimes \mu_q\otimes\nu_1\otimes\cdots\otimes \nu_p]
            \end{align*}
            where 
            \begin{align*}
                (f_1\times f_2)(i) = \begin{cases}
                    f_1(i) & \text{if $1\leq i\leq q$}\\
                    f_2(i-q) & \text{if $q+1\leq i\leq q+p$}.
                \end{cases}
            \end{align*}
            
        \end{lemma} 
            \indent Now that we have all the main ingredients, we can define the free operad associated to a $\Sigma$-module $E$.  
        \begin{definition}
            For any $\Sigma$-module $E$, define $(F(E), \circ, \eta)$ as follows. 
            \begin{itemize}
                \item For each $n\geq 1$, define
                    $$F(E)(n) = \bigoplus_{[T]\in Tree(n)} E[T]$$
                    with right $\Sigma_n$-action induced by $M(\sigma^*):M[\sigma^*(T)]\rightarrow M[T]$ for each $\sigma\in \Sigma_n.$
                \item Define the unit $\eta:k\rightarrow F(E)(1)$ as the following composition
                $$\begin{tikzcd}
                 k\arrow[r, "\eta"]\arrow[d, "="] & F(E)(1)=E[\downarrow]\\
                 E[\downarrow]\arrow[ur, "\text{inclusion}"']
                \end{tikzcd}$$
                \item Define the $\circ_i$ composition as follows. Since $\otimes$ commutes with direct sums on each side, then we have the natural right $\Sigma_n\times \Sigma_m$ equivariant isomorphism
                \begin{align*}(F(E))(n)\otimes (F(E)(m)) &= \left(\bigoplus_{[T_1]\in Tree(n)} E[T]\right) \otimes \left( \bigoplus_{[T_2]\in Tree(m)} E[T_2]\right)\\
                &\cong \bigoplus_{([T_1],[T_2])\in Tree(n)\times Tree(m)} E[T_1]\otimes E[T_2].\end{align*}
                So it suffices to define $\circ_i$ restricted to a direct summand 
                $$\begin{tikzcd}
                    E[T_1]\otimes X[T_2]\arrow[r, "\Psi"]\arrow[d,"\text{inclusion}"] & X[T_1\circ_e T_2]\arrow[d,"\text{inclusion}"]\\
                    F(E)(n)\otimes F(E)(m)\arrow[r, "-\circ_i-"] & F(E)(n+m-1)
                \end{tikzcd}$$
                where $e$ is the is the $i$th edge of $T_1$ and $T_1\circ_e T_2$ is the grafting of $T_1$ and $T_2$ along $e$.  
                \end{itemize}
        \end{definition}

        \indent There is a natural weight grading on the free operad $F(E)$ as follows: for $n\geq 0$, define a $\Sigma$-submodule $F(E)^{(n)}$ of $F(E)$ with 
            \begin{align*}
               F(E)^{(n)}(m) = \bigoplus_{[T]\in Tree_n(m)} E[T] 
            \end{align*}
            where $Tree_n(m)$ is the equivalence classes of $m$-rooted trees with $n$ vertices.

\section{$n$-Quadratic Operads}

    \subsection{Description of $n$-Quadratic Operad}
        \indent  Recall from \cite{loday}, they define quadratic data as a pair $(E,R)$ with $E$ is a $\Sigma$-module and $R\subseteq F(E)^{(2)}$.  For our context, we are interested in a particular type of quadratic operads, ones that are generated in arity $n\geq 2$.  We say a quadratic data $(E,R)$ is a $n$-quadratic data if $E$ is concentrated in arity $n$ and $R$ is a $\Sigma$-submodule of $F(E)^{(2)}$.  
        \\
        \indent A $n$-quadratic operad is a quadratic operad $\mathcal{P}(E,R)=F(E)/(R)$ for $n$-quadratic data $(E,R)$.  For $n=2$, these are exactly the binary quadratic operads as studied by Ginzburg and Kapranov in their study of Koszul duality in \cite{ginzburg}. 

    \subsection{Description of $F(E)^{(2)}$}
    
       \indent This section, we will want to find a useful description for the weighted $2$ part of a $n$-quadratic operad $\cP(E,R)$, to better understand how to find the Koszul dual.
        \\
        \indent Since our free operads are defined by rooted trees, we need some description for the rooted trees with exactly $2$-vertices and the symmetric group action on them.  Let $Tree_2(n,m)$ be the subset of $Tree_2(n+m-1)$ consisting of isomorphism classes of $[T]$ with exactly two vertices $v_1$ and $v_2$ such that $v_2\in In_T(v_1)$ and $|in_T(v_1)|=n$ and $|in_T(v_2)|=m$.  We have the following description of $Tree_2(n,m)$ in terms of the left cosets of $\Sigma_m$ in $\Sigma_{n+m-1}$.

        \begin{lemma}\label{2-tree decomposition}
                \indent The set $Tree_2(n,m)\cong \Sigma_{n+m-1}/(\overline{\Sigma_{n-1}}\times\Sigma_m)\times \{[Cor_n\circ_1Cor_m]\}$, where $\Sigma_{n+m-1}/(\overline{\Sigma_{n-1}}\times\Sigma_m)$ is the collection of left cosets of $\overline{\Sigma_{n-1}}\times \Sigma_m$.  
            \end{lemma}
            \begin{proof}
               \indent Any $n+m-1$-tree with two vertices $v_1$ and $v_2$ such that $v_2\in in_T(v_1)$ and $|int_T(v_1)|=n$ and $|in_T(v_2)|=m$ is isomorphic to the tree $\sigma^*(Cor_n\circ_1Cor_m)$ for some permutation $\sigma\in \Sigma_{n+m-1}$ that preserves the order of the inputs, and where $Cor_n\circ_1Cor_m$ has its inputs for the vertices ordered from $1$ to $m$ for the inputs into $Cor_m$ and $m+1$ to $m+n-1$ for the inputs of $Cor_n$ excluding the one that is connected to the root of $C_m$.

               To do this, take our tree $T$ and we have 
                \begin{align*}
                    in_T(v_1)&=\{v_2,i_1,\dots, i_{n-1}\}\\
                    in_T(v_2)&= \{j_1,\dots, j_m\}
                \end{align*}
                where $i_1,\dots, i_{n-1},j_1,\dots, j_m$ are distinct elements of $[n+m-1]$.  We can define a permutation $\sigma$ such that $\sigma(\{i_1,\dots, i_{n-1}\}=\{m+1,\dots, m+n-1\}$ and $\sigma(\{j_1,\dots, j_m\})=\{1,\dots, m\}$ and its ordered i.e. if $i_p<i_q$ then $\sigma(i_p)<\sigma(i_q)$. Hence, we have the tree $\sigma^*(T)\cong Cor_n\circ_1 Cor_m$ and $T\cong(\sigma^{-1})^*(Cor_n\circ_1Cor_m)$.
                \\
                \indent Next, let $\sigma_1=1,\sigma_2,\dots, \sigma_l$ be the $l=\frac{(n+m-1)!}{(n-1)!m!}$ representatives for the left cosets of $\Sigma_{n-1}\times\Sigma_m$ in $\Sigma_{n+m-1}$.  Take any tree $T$ in $Tree_2(n,m)$, and by above, $T=\tau^*(Cor_n\circ_1Cor_m)$ for some permutation $\tau\in \Sigma_{n+m-1}$.  Then we have $\tau=\sigma_i\omega$ for some $1\leq i\leq n$ and $\omega\in \Sigma_{n-1}\times\Sigma_m$.  Therefore, we have
                \begin{align*}
                    T=\tau^*(Cor_n\circ_1Cor_m) = (\sigma_i\omega)^*(Cor_n\circ_1Cor_m) = \sigma_i^*(Cor_n\circ_1Cor_m).  
                \end{align*}
                This shows that $Tree_2(n,m) = \Sigma_{n+m-1}/(\overline{\Sigma_{n-1}}\times\Sigma_m)\times \{[Cor_n\circ_1Cor_m]\}.$
            \end{proof}
        
        \indent With our description for $Tree_2(n,m)$ we can push this through and find a useful description for $F(E)^{(2)}$ in the case where $E$ is concentrated in arity $n=m$.  
        \begin{lemma}\label{description of weight 2}
            \indent Let $E$ be a $\Sigma$-module concentrated in arity $n$.  Then the sub $\Sigma$-module $F(E)^{(2)}$ has the following description:
            \begin{align*}
                F(E)^{(2)}(p) = \begin{cases}
                    \bigoplus_{[T]\in Tree_{2}(p)} E[T] & \text{if $p=2n-1$}\\
                    0 &\text{otherwise}
                \end{cases}
            \end{align*}
            and 
            \begin{align*}
                F(E)^{(2)}(2n-1) &\cong \bigoplus_{i=1}^l (E(n)\otimes E(n))\sigma^{-1}_i\\
                &= \text{Ind}_{\overline{\Sigma_{n-1}}\times\Sigma_n}^{\Sigma_{2n-1}}(E(n)\otimes E(n))
            \end{align*}
        \end{lemma}
        for representative $\sigma_1=1,\dots, \sigma_l$ for $l=\frac{(2n-1)!}{(n-1)!n!}$ of the left cosets $\overline{\Sigma_{n-1}}\times \Sigma_n$.
        \begin{proof}
            \indent By definition
            \begin{align*}
                F(E)^{(2)}(2n-1) &= \bigoplus_{[T]\in Tree_2(n,n)} E[T],\\
            \end{align*}
            since if a tree with two vertices did not have a vertex with $n$ leaves, then an $E$-decoration of that tree would be zero.  
            Pick representatives $\sigma_1=1,\sigma_2,\dots, \sigma_l$ for $l=\frac{(2n-1)!}{(n-1)!n!}$ and by lemma \ref{2-tree decomposition}, we have 
            \begin{align*}
                \bigoplus_{[T]\in Tree_2(n,n)}E[T]  &= \bigoplus_{i=1}^l E[\sigma_i^*(Cor_n\circ_1Cor_n)]\\
                &\cong \bigoplus_{i=1}^l E[Cor_n\circ_1Cor_n] \sigma_i^{-1}
            \end{align*}
            is isomorphic to the induced representation of $\Sigma_{2n-1}$ over the subgroup $\overline{\Sigma_{n-1}}\times\Sigma_n$ and where $E[Cor_n\circ_1Cor_n]$ has the right $\overline{\Sigma_{n-1}}\times\Sigma_n$ action in the natural way.  
            \\
            \indent As right $\overline{\Sigma_{n-1}}\times\Sigma_n$-module,
            \begin{align*}
                E[Cor_n\circ_1Cor_n] = E(in(v_1))\otimes E(in(v_2))
            \end{align*}
            where $v_1$ and $v_2$ are the vertices with $in(v_1)=\{v_2,n+1,\dots, 2n-1\}$, $in(v_2)=[n]$ and $in(r_{Cor_n\circ_1Cor_n})=\{v_1\}$.  We have
            \begin{align*}
                E(in(v_1))\otimes E(in(v_2))\cong E(n)\otimes E(n)
            \end{align*}
            as right $\overline{\Sigma_{n-1}}\times\Sigma_n$-modules.  Finally we obtain
            \begin{align*}
                \bigoplus_{[T]\in Tree_2(n,n)} E[T]&\cong \bigoplus_{i=1}^l (E(n)\otimes E(n))\sigma^{-1}_i\\
                &= \text{Ind}_{\overline{\Sigma_{n-1}}\times\Sigma_n}^{\Sigma_{2n-1}}(E(n)\otimes E(n))
            \end{align*}
            This completes the proof.
        \end{proof}
        \indent The elements in $E$ in $F(E)/(R)$ can be thought of as operations $\mu$ which act on formal elements $x_1,\dots, x_n$ through $\mu(x_1,\dots, x_n)$.  Therefore, elements of $F(E)^{(2)}(2n-1)$ can be thought of as elements $\mu\circ_{\sigma}\nu$ which act on $2n-1$ formal elements $x_1,\dots, x_{2n-1}$ where 
        \begin{align*}
            (\mu\circ_{\sigma}\nu)(x_1,\dots, x_{2n-1}) &= (\mu\circ_1\nu)^{\sigma}(x_1,\dots, x_{2n-1})\\
            &= (\mu\circ_1\nu)(x_{\sigma(1)},\dots, x_{\sigma(2n-1)})\\
            &=\mu(\nu(x_{\sigma(1)},\dots, x_{\sigma(n)}),x_{\sigma(n+1)},\dots, x_{\sigma(2n-1)}).
        \end{align*}

        \subsection{Koszul Dual of $n$-Quadratic Operad}\label{sec: koszul dual}
        \indent For this subsection, we will describe how to find the Koszul dual for any $n$-quadratic operad using some bilinear pairing on the weighted 2 part of the free operad by following the ideas in \cite{Markl2015}, where they generalized the Koszul duality of binary operads to quadratic operads with $n$-arity generators.  
        \\
        \indent For any $\Sigma$-module $E$, define $E\otimes Sgn$ to be the $\Sigma$-module with $(E\otimes Sgn)(n)=E(n)\otimes Sgn_n$, where $Sgn_n$ is the signed representation of $\Sigma_n$.  Furthermore, let $E^*$ be the $\Sigma$-module with $E^*(n)=E(n)^*$, the $k$-linear dual with the induced right $\Sigma_n$-action.  Combining these two definitions, define $E^{\vee}$ to be the $\Sigma$-module with $E^{\vee}(a) = \uparrow^{a-2}E^*(a)\otimes Sgn_a$, where $\uparrow^{a-2}$ denotes the suspension iterated $a-2$ times, which is naturally a $\Sigma$-module. More explicitly the action of $\sigma_n$ on the right is defined as follows: if $f\in \uparrow^{n-2}E^*(n)$, $v\in E(n)$, and $\sigma\in \Sigma_n$, then
        \begin{align*}
            f^{\sigma}(v) = Sgn(\sigma) f(v^{\sigma^{-1}}).
        \end{align*}
        \indent Hence, if $E$ is concentrated in arity $n$ and in degree $d$, then $E^{\vee}$ is concentrated in arity $n$ and in degree $-d+n-2$.

        \begin{definition}
            \indent A pairing between two $\Sigma$-modules $M$ and $N$, is a $\Sigma$-module map $$\langle -,-\rangle: M\otimes_HN\rightarrow Com\otimes Sgn.$$
            \\
            \indent If $L$ is a sub $\Sigma$-module of $N$, then $L^{\perp}$ is the sub $\Sigma$-module of $M$ with $L^{\perp}(n)=\{ x\in M~:~\langle x,L(n)\rangle=0\}$.
        \end{definition}

        \indent Explicitly, the pairing has the following relation: if $f\in M(n)$, $g\in N(n)$, and $\sigma\in \Sigma_n$, then
        \begin{align*}
            \langle f^{\sigma},g^{\sigma}\rangle = Sgn(\sigma)\langle f,g\rangle.
        \end{align*}

        \indent
        \\
        \indent Let $(E,R)$ be $n$-quadratic data.  Following the binary quadratic operad situation, we want to define a non-degenerate bilinear pairing $\langle -,-\rangle: F(E^{\vee})^{(2)}\otimes_H F(E)^{(2)}\rightarrow Com\otimes Sgn$.  
        To start, let $S^{2n-1}_n$ be a collection of representatives of the right cosets of $\overline{\Sigma_{n-1}}\times\Sigma_n$ in $\Sigma_{2n-1}$ consisting of even permutations.  By lemma \ref{description of weight 2}, we have the isomorphisms
        \begin{align*}
           \text{Ind}_{\overline{\Sigma_{n-1}}\times\Sigma_n}^{\Sigma_{2n-1}}(E(n)\otimes E(n))\cong \bigoplus_{\sigma\in S_n^{2n-1}} E(n)\otimes E(n)\sigma
        \end{align*}
        and we will denote our elements in $\text{Ind}_{\overline{\Sigma_{n-1}}\times\Sigma_n}^{\Sigma_{2n-1}}(E(n)\otimes E(n))$ as
        $
            \mu\circ_{\sigma}\nu
       $ 
        to represent the image of the element $(\mu\otimes \nu)\sigma$ for $\mu\in E(n)$, $\nu\in E(n)$ and $\sigma\in S_{n}^{2n-1}$.  If $\tau\in \Sigma_{2n-1}$, then we define
        \begin{align*}
            (\mu\circ_{\sigma}\nu)^{\tau} = \mu^{y}\circ_{\omega} \nu^h
        \end{align*}
        for some $y\times h\in \overline{\Sigma_{n-1}}\times\Sigma_{n}$ and $\omega\in S_{n}^{2n-1}$ such that $(y\times h)\omega=\sigma\tau$. For $\alpha^*\circ_{\sigma}\beta^*\in F(E^{\vee})^{(2)}(2n-1)$ and $\mu\circ_{\tau}\nu\in F(E)^{(2)}(2n-1)$, define
        \begin{align*}
            \langle \alpha^*\circ_{\sigma}\beta^*, \mu\circ_{\tau}\nu\rangle = \begin{cases}
                \alpha^*(\mu)\beta^*(\nu) & \text{if $\sigma=\tau$}\\
                0 & \text{otherwise}
            \end{cases}
        \end{align*}
        This is exactly the same bilinear form as defined in \cite{Markl2015}, where here we are only composing the elements at the first component, which does not produce any extra signs.  

        \begin{lemma}
            The pairing $\langle -,-\rangle :F(E^{\vee})\otimes_HF(E)^{(2)}\rightarrow Com\otimes Sgn$ is a non-degenerate bilinear pairing of $\Sigma$-modules.
        \end{lemma}
        \begin{proof}
            \indent It is obvious that this is non-degenerate since it is defined through the basis elements, so it suffices to show that it respects the action by permutations.  Let $\tau\in \Sigma_{2n-1}$, $\alpha^*\circ_{\sigma}\beta^*\in F(E^{\vee})^{(2)}(2n-1)$, and $\mu\circ_{\omega}\nu\in F(E)^{(2)}(2n-1)$, and by definition we have
            \begin{align*}
                \langle (\alpha^*\circ_{\sigma}\beta^*)^{\tau}, (\mu\circ_{\omega}\nu)^{\tau})\rangle &=\langle (\alpha^*)^{y_{\sigma}}\circ_{\sigma'} (\beta^*)^{h_{\sigma}}, \mu^{y_{\omega}}\circ_{\omega'}\nu^{h_{\omega}}\rangle 
            \end{align*}
            where $(y_{\sigma}\times h_{\sigma})\sigma'=\sigma\tau$ and $(y_{\omega}\times h_{\omega})\omega'=\omega\tau$ for some $h_{\sigma},h_{\omega}\in \Sigma_{n}$, $y_{\sigma},y_{\omega}\in\overline{\Sigma_{n-1}}$, and $\sigma',\omega'\in \Sigma_{2n-1}$.
            \\
            \indent By definition of our pairing,
            \begin{align*}
               \langle \alpha^*\circ_{\sigma'} (\beta^*)^{h_{\sigma}}, \mu\circ_{\omega'}\nu^{h_{\omega}}\rangle &=\begin{cases}
                   Sgn(y_{\sigma})Sgn(h_{\sigma})\alpha^*(\mu^{y_{\omega}y_{\sigma}^{-1}}) \beta^*(\nu^{h_{\omega}h_{\sigma}^{-1}}) & \text{if $\sigma'=\omega'$}\\
                   0 & \text{otherwise}
               \end{cases} 
            \end{align*}
            In the case where $\sigma'=\omega'$, we have
            \begin{align*}
            (y_{\omega}^{-1}\times h_{\omega}^{-1})\omega\tau=\omega'=\sigma'=(y_{\sigma}^{-1}\times h_{\sigma}^{-1})\sigma\tau,\\
            \end{align*}
            which implies $(y_{\sigma}\times h_{\sigma})(y_{\omega}^{-1}\times h_{\omega}^{-1})\omega=\sigma$.
            This says that $\omega$ and $\sigma$ are representatives of the same right coset, hence they are equal.  Therefore, we obtain $h_{\omega}=h_{\sigma}$ and $y_{\sigma}=y_{\omega}$, and this gives
            \begin{align*}
               \langle \alpha^*\circ_{\sigma'} (\beta^*)^{h_{\sigma}}, \mu\circ_{\omega'}\nu^{h_{\omega}}\rangle &=\begin{cases}
                   Sgn(y_{\sigma})Sgn(h_{\sigma})\alpha^*(\mu^{y_{\omega}y_{\sigma}^{-1}}) \beta^*(\nu^{h_{\omega}h_{\sigma}^{-1}}) & \text{if $\sigma'=\omega'$}\\
                   0 & \text{otherwise}\\
                \end{cases}\\
                &=\begin{cases}
                    Sgn(y_{\sigma})Sgn(h_{\sigma})\alpha^*(\mu) \beta^*(\nu) & \text{if $\sigma'=\omega'$}\\
                    0 & \text{otherwise}
                \end{cases}\\
                &=Sgn(y_{\sigma})Sgn(h_{\sigma})\langle \alpha^*\circ_{\sigma}\beta^*, \mu\circ_{\omega}\nu\rangle.
            \end{align*}
            Since $(y_{\sigma}\times h_{\sigma})\sigma'=\sigma\tau$ and both $\sigma$ and $\sigma'$ are even permutations, then $Sgn(y_{\sigma})Sgn(h_{\sigma})=Sgn(y_{\sigma}\times h_{\sigma})=Sgn(\tau)$.  This completes the proof.
        \end{proof}

        \indent As in the binary quadratic operad we have the same description for the Koszul dual for a $n$-quadratic operad, see \cite{loday} and \cite{Markl2015}.

        \begin{definition}
             \indent Let $\cP=\cP(E,R)$ be a $n$-quadratic operad.  Then its Koszul dual $\cP^{\text{!}}$ is the $n$-quadratic operad
            \begin{align*}
                \cP^{\text{!}} = \cP(E^{\vee},R^{\perp}).
            \end{align*}
        \end{definition}  

        \subsubsection{(anti)-Commutative $n$-Quadratic Operads and their Koszul dual}
            \indent The operads we are interested are the ones with $n$-ary operations that have either commutative or anti-commutative operations concentrated in arity $n$ and in degree $d$, like the $n$-Lie algebras and $n$-Com algebras of degree $d$.  In this scenario, the weight $2$ part of the free operad has a particularly useful description suitable for computations.  To help with the cluttering of symbols, define $\Sigma_{n-1,n}=\overline{\Sigma_{n-1}}\times\Sigma_n$.  
            \\
            \indent Let $n\geq 2$ for this section. In this scenario, the collection of right cosets of $\Sigma_{n-1,n}$ in $\Sigma_{2n-1}$ has very useful property with respect to $\Sigma_n$ and $\overline{\Sigma}_{n-1}$, separately.  Explicitly, if $\sigma,\tau\in \Sigma_{2n-1}$, then $\Sigma_{n-1,n} \sigma = \Sigma_{n-1,n}\tau$ if and only if $\sigma^{-1}(i) = \tau^{-1}(h(i))$ and $\sigma^{-1}(j)=\tau^{-1}(y(j))$ for some $y\times h\in \overline{\Sigma}_{n-1}\times \Sigma_n$ and for $1\leq i\leq n$ and $n+1\leq j\leq 2n-1$.  This shows that the distinct right cosets of $\overline{\Sigma}_{n-1}\times \Sigma_n$ are entirely based on where the inverse of the representatives send the numbers $n+1,\dots, 2n-1$, up to any permutation in $\overline{\Sigma}_{n-1}$.  
            \\
            \indent Let $a_1,\dots, a_{n-1}$ be distinct elements of $\{1,\dots, 2n-1\}$ and define $(\Sigma_{n-1,n})\{a_1,\dots, a_{n-1}\}$ to be the collection of permutations $\sigma\in\Sigma_{2n-1}$ such that 
            \begin{align*}
                \sigma^{-1}(\{n+1,\dots, 2n-1\})=\{a_1,\dots, a_{n-1}\}.
            \end{align*}  
            Let $\Omega_n^{2n-1}$ be the collection of sets $(\Sigma_{n-1,n})\{a_1,\dots, a_{n-1}\}$ as defined above, which is equivalent to the set of right cosets of $(\overline{\Sigma}_{n-1}\times \Sigma_n)$ by sending $(\Sigma_{n-1,n})\sigma\mapsto (\Sigma_{n-1,n})\{\sigma^{-1}(n+1),\dots, \sigma^{-1}(2n-1)\}$.  Denote by $\Lambda_n^{2n-1}$ to be the set of finite sets $\{a_1,\dots, a_{n-1}\}$ for distinct elements $a_1,\dots, a_{n-1}\in \{1,\dots, 2n-1\}$.      
   The natural action of $\Sigma_{2n-1}$ on the right cosets of $\Sigma_{n-1,n}$ transfers to a natural action on the collection $\Omega_n^{2n-1}$ in the following way: for $\sigma\in \Sigma_{2n-1}$, define $(\Sigma_{n-1,n})\{a_1,\dots, a_{n-1}\}\sigma=(\Sigma_{n-1,n})\{\sigma^{-1}(a_1),\dots, \sigma^{-1}(a_{n-1})\}.$
        \\
        \indent Let $M$ be $\Sigma_{n-1,n}$-module.  If $M$ is either the signed representation or the trivial representation and we choose representatives for the right cosets of $\Sigma_{n,n-1}$ to be even permutations, then
        \begin{align*}
            Ind_{\Sigma_{n-1,n}}^{\Sigma_{2n-1}}(M)\cong \bigoplus_{\{a_1,\dots, a_{n-1}\}\in \Lambda_{n}^{2n-1}} M\{a_1,\dots, a_{n-1}\}
        \end{align*}
        where $M\{a_1,\dots, a_{n-1}\}$ are isomorphic to $M$ as graded $k$-modules, with elements denoted as $e \{a_1,\dots, a_{n-1}\}$ for $e\in M$ and $\{a_1,\dots, a_{n-1}\}\in \Lambda_{n}^{2n-1}$.  The isomorphism above is as right $\Sigma_{2n-1}$ modules with the action on the right as 
        \begin{align*}
            e\{a_1,\dots, a_{n-1}\}\sigma = \begin{cases}
                e\{\sigma^{-1}(a_1),\dots, \sigma^{-1}(a_{n-1})\} & \text{if $M$ is the trivial representation}\\
                sgn(\sigma) e\{\sigma^{-1}(a_1),\dots, \sigma^{-1}(a_{n-1})\} & \text{if $M$ is the signed representation}
            \end{cases}
        \end{align*}
        for $\sigma\in \Sigma_{2n-1}$.

        \indent
        \\
        \indent In particular, if $(E,R)$ is a $n$-quadratic data where $E(n)$  either has the trivial or signed action, then we have the following description for $F(E)^{(2)}:$
        \begin{align}\label{eq: Ind description}
           \text{Ind}_{\Sigma_{n-1,n}}^{\Sigma_{2n-1}}(E(n)\otimes E(n))\cong \bigoplus_{\{a_1,\dots, a_{n-1}\}\in \Lambda_{n}^{2n-1}} E(n)\otimes E(n)\{a_1,\dots, a_{n-1}\}.
        \end{align}
        \indent Following the general situation, we will denote our elements as
        \begin{align*}
            \mu\circ_{\{a_1,\dots, a_{n-1}\}}\nu
        \end{align*}
        to represent the image of the element $\mu\otimes \nu[a_1,\dots, a_{n-1}]$. Furthermore, we have a non-degenerate bilinear pairing $\langle -,-\rangle: F(E^{\vee})^{(2)}\otimes F(E)^{(2)}\rightarrow Com\otimes Sgn$ defined in the same way as in the general situation:
        \begin{align*}
            \langle \alpha^*\circ_{\{a_1,\dots, a_{n-1}\}}\beta^*,\mu\circ_{\{b_1,\dots, b_{n-1}\}}\nu\rangle = \begin{cases}
                \alpha^*(\mu) \beta^*(\nu) & \text{if $\{a_1,\dots, a_{s-1}\}=\{b_1,\dots, b_{u-1}\}$}\\
                0 & \text{otherwise}
            \end{cases}
        \end{align*}.

\section{$n\text{-}Lie_d$ and $n\text{-}Com_d$ Operads}
    \indent We are now ready to construct our operads of interest, $n\text{-}Lie_d$ and $n\text{-}Com_d$.  We will first describe the algebras for these operads, give a few examples and construct their quadratic representation for them.
     
    \subsection{General Definitions for $n$-Lie algebras of degree $d$}
        
         \begin{definition}\label{def of n lie}
            \indent Let $n\geq 2$.  An $n$-Lie algebra of degree $d$ is a graded $k$-module $L$ with a $n$-ary operation $l_n:L^{\otimes n}\rightarrow L$ of satisfying the following properties:
            \begin{itemize}
                \item for any $n$ elements $v_1,\dots, v_n\in L$, and any $\sigma\in \Sigma_n$, 
                \begin{align*}
                    Sgn(\sigma)[v_1,\dots, v_n] = \xi(\sigma,v_1,\dots, v_n)[v_{\sigma(1)},\dots, v_{\sigma(n)}]
                \end{align*}
                \item and for any $2n-1$ elements $v_1,\dots, v_{2n-1}\in L$, the $n$-ary operation satisfies the following generalized Jacobi-identity:
            \begin{align*}
                [[v_1,\dots, v_n],v_{n+1},\dots, v_{2n-1}] = \sum_{i=1}^{n}(-1)^{\varepsilon_i}[v_1,\dots, v_{i-1},[v_i,v_{n+1},\dots, v_{2n-1}],v_{i+1},\dots, v_n]
            \end{align*}
            where 
            \begin{align*}
                \varepsilon_i = d(\sum_{j=1}^{i-1}|v_j|) + (\sum_{j=i+1}^n|v_j|)(\sum_{r=n+1}^{2n-1}|v_r|).
            \end{align*}
            \end{itemize}
        \end{definition}
        \indent In this definition, if $d=0$ and the underlying algebra is concentrated in degree $0$, then we will call these just $n$-Lie algebras.  
        \indent 
        \\
        \begin{example}
            \begin{itemize}
                \item Let $A=k[x_1,\dots, x_{n+1}]$ be the polynomial ring on $n+1$-variables.  Then we can define a $n+1$-ary bracket on the elements $p_1,\dots, p_{n+1}\in A$ 
                \begin{align*}
                   [p_1,\dots, p_{n+1}]=\text{Jac}(p_1,\dots, p_{n+1}) 
                \end{align*}
                where the right hand side is the determinant of the associated Jacobian matrix of the polynomials $p_1,\dots, p_{n+1}$.  This is naturally a $n+1$-Lie algebra and of great interest, specifically for $n=3$ in Nambu mechanics. 
                \\
                \indent Next, one can pick element $\Omega\in A$, which is called the potential, and define a $n$-ary bracket
                \begin{align*}
                    [p_1,\dots, p_n]_{\Omega} = [p_1,\dots, p_n,\Omega]
                \end{align*}
                which becomes a $n$-Lie algebra.
                \item More generally, suppose $A$ is a commutative $k$-algebra with commuting derivations $D_1,\dots, D_n$.  Then $A$ is a $n$-Lie algebra with bracket
                \begin{align*}
                    [a_1,\dots, a_n] = \frac{1}{n!}\sum_{\sigma\in\Sigma_n}Sgn(\sigma) D_{\sigma(1)}(a_1)\cdots D_{\sigma(n)}(a_n) = \text{det}\begin{pmatrix}
                    D_1(a_1) & \cdots & D_1(a_n)\\
                    \vdots & \cdots & \vdots\\
                    D_n(a_1) & \cdots & D_n(a_n).
                    \end{pmatrix}
                \end{align*}
                \item In Filippov's study of $n$-Lie algebras and their algebraic properties in \cite{filippov}, they defined the following example which helped in the classifications for simple $n$-Lie algebras of finite dimensions.  Let $L$ be a $n+1$ dimensional vector space with basis $\{v_1,\dots, v_{n+1}\}$.  Define the $n$-bracket
                \begin{align*}
                    [v_1,\dots, \widehat{v_i},\dots, v_{n+1}] = (-1)^{n+1+i}v_i
                \end{align*}
                which gives $L$ a $n$-Lie algebra structure. 
            \end{itemize}
        \end{example}
        \indent
        \\ 
            \indent For the next lemma, we want to express the defining equation for a $n$-Lie algebras of degree $d$ using just the operation and permutation actions.

        \begin{lemma}\label{relations of n lie}
            \indent Let $L$ be a $n$-Lie algebra of degree $d$ and let $l_n=[-,\dots, -]$ be the $n$-ary bracket on $L$.  If $\lambda=(1,2,\dots, 2n-1)$ is the standard $2n-1$-cycle and $\omega=(1,2,\dots, n)$ is the standard $n$-cycle in $\Sigma_{2n-1}$, then the defining relation for a $n$-Lie algebra of degree $d$ is
            \begin{align}\label{equation n-Lie}
                l_n\circ_1l_n + (-1)^n\sum_{i=0}^{n-1}(-1)^{i(n+1)}(l_n\circ_1l_n)^{\lambda^n\omega^i}=0
            \end{align}
            where $f\circ_i g$ is defined in operad example \ref{ex: operads}.
        \end{lemma}
        \begin{proof}
            \indent For ease of calculations, here are the expressions for $\lambda^n\omega^i$ for $0\leq i\leq n-1$:
            \begin{align*}
                \lambda^n\omega^i=\begin{cases}
                    \begin{pmatrix}1 & 2 & \cdots & n-1 & n & n+1 & \cdots & 2n-1\\ n+1 & n+2 & \cdots & 2n-1 & 1 & 2 & \cdots & n \end{pmatrix}& \text{$i=0$}\\
                    \begin{pmatrix} 1 & 2 & \cdots & n-i-1 & n-i & n-i+1 & \cdots &n & n+1 & \cdots\\ n+i+1 & n+i+2 & \cdots & 2n-1 & 1 & n+1 & \cdots & n+i & 2&\cdots   \end{pmatrix} & \text{$1\leq i\leq n-2$}\\
                    \begin{pmatrix}1 & 2 & 3 &\cdots & n & n+1 & n+2 & \cdots & 2n-1\\ 1 & n+1 & n+2 & \cdots & 2n-1 & 2 & 3 
 & \cdots & n \end{pmatrix} & \text{$i=n-1$}
                \end{cases}
            \end{align*}
            \indent Let $v_1,\dots, v_{2n-1}\in L$, then the equation \ref{equation n-Lie} applied to these elements give us
            \begin{align*}
                &(l_n\circ_1l_n)(v_1,\dots, v_{2n-1})+(-1)^n\sum_{i=0}^{n-1}(-1)^{i(n+1)}(l_n\circ_1l_n)^{\lambda^n\omega^i}(v_1,\dots, v_{2n-1})\\
                &=l_n(l_n(v_1,\dots, v_n),v_{n+1},\dots, v_{2n-1})\\
                &+(-1)^n\xi(\lambda^n,v_1,\dots, v_{2n-1})l_n(l_n(v_n,\dots, v_{2n-1}),v_1,\dots, v_{n-1})\\
                &+(-1)^n\sum_{i=1}^{n-2}(-1)^{i(n+1)}\xi(\lambda^n\omega^i,v_1,\dots, v_{2n-1})l_n(l_n(v_{n-i},v_{n+1},v_{n+2},\dots, v_{2n-1}),v_{n-i+1},\dots, v_n,v_1,\dots, v_{n-i-1})\\
                &+(-1)^n\xi(\lambda^n\omega^{n-1},v_1,\dots, v_{2n-1})l_n(l_n(v_1,v_{n+1},\dots, v_{2n-1}),v_2,\dots, v_n)
            \end{align*}
            For each of these terms in the equation above, if we move the $l_n(*,\dots, *)$ term to the right to their appropriate place, it will cancel the appropriate permuted terms in $\xi(\lambda^n\omega^i,v_1,\dots, v_{2n-1})$ and we keep the terms
            \begin{align*}
                \left(\sum_{j=i+1}^n|v_j|\right)\left(\sum_{r=n+1}^{2n-1}|v_r|\right),
            \end{align*}
            and since we are moving a degree $d$ operation $l_n$ as well, we add in $d\left(\sum_{j=1}^{n-i-1}|v_j|\right)$.  This will give us exactly the relations for a $n$-Lie algebra of degree $d$.

        \end{proof}
        \indent
        \\
        \indent In the last lemma, note that if $n$ is odd, then $\lambda^n\omega^i$ is always an even permutation in this case and we obtain the relation
        \begin{align*}
            (l_n\circ_1l_n) - \sum_{i=0}^n (l_n\circ_1l_n)^{\lambda^n\omega^i}=0
        \end{align*}
        In the case when $n$ is even, we have that $\lambda^n\omega^i$ is odd if and only if $i$ is odd.  Therefore, to make sure that we can use these equations in in our construction of our operads, we need to make sure the permutations acting on our operations are even.  Therefore, in the case when $n$ is even, we have
        \begin{align*}
            (l_n\circ_1l_n) +\sum_{i=0}^{n-1}(-1)^{i}(l_n\circ_1l_n)^{\lambda^n\omega^i}
            &= (l_n\circ_1l_n) + \sum_{i=0}^{n-1}(-1)^i(l_n\circ_1l_n)^{(1~2)^i(1~2)^i\lambda^n\omega^i}\\
            &=(l_n\circ_1l_n) + \sum_{i=0}^{n-1}(-1)^i(-1)^i(l_n\circ_1l_n)^{(1~2)^i\lambda^n\omega^i}\\
            &=(l_n\circ_1l_n) + \sum_{i=0}^{n-1}(l_n\circ_1l_n)^{(1~2)^i\lambda^n\omega^i}=0\\
        \end{align*}
        where we used the fact that $(1~2)$ acting on $l_n$ is $-l_n$.  Hence, we have a representation of the relations for $n-Lie$ algebras of degree $d$ using even permutations.
        \\
        \indent The relations in lemma \ref{relations of n lie} can be thought of as some sort of generalized Garnir relation as explained in \cite{FRIEDMANN2021107570}.  This gives some sort of explanation of why the relations for the Koszul dual of $n\text{-}Lie_d$ would be coming from Young symmetrizer relations as shown in lemma \ref{snsubspace}.

        \subsection{Definition of $n\text{-}Lie_d$ Operad}
                \indent Before we define our operad $n\text{-}Lie_d$, we need to define our collection of trees that we will use, that will generate our operad.
            \begin{definition}
                \indent Let $r\geq 0$ and $n\geq 2$.  Define $Tree_p(r,n)$ to be the collection of $n$-ary planar rooted trees $T$ with at most $r$ vertices, equipped with a input labeling $\Lambda_T:int_T\rightarrow [rn-r+1]$, that respects the planar structure $\{\Psi_v~:~[|in_T(v)|]\rightarrow in_T(v)\}_{v\in V_T}$.
                \\
                \indent Define $Tree_p(r,n)_s$ for $1\leq s\leq r$ to be the subset of $Tree_p(r,n)$ consisting of trees with exactly $s$ vertices and define $Tree_p(r,n)_0$ to just be the set of trees $\downarrow$ with the input labeled by an element in $\{1,\dots, rn-r+1\}$.  This decomposes the set $Tree_p(r,n)$ as a disjoint union of $Tree_p(r,n)_s$ for $0\leq s\leq r$.  
            \end{definition}
            
            \indent Every tree $T\in Tree_p(r,n)_s$ can be represented by $T=[T_1,\dots, T_n]$, where $T_i\in Tree_p(r,n)_{s_i}$ for $s_1+\cdots + s_n=r-1$ and $T_i$ is the $i$th subtree of $T$, representing the grafting of these trees on the $i$th leaf of $Cor_n$.  There is a natural correspondence between the trees $T\in Tree_p(r,n)$ and at most $r$ $n$-ary brackets of elements $\{1,\dots, rn-r+1\}$.  For example, suppose we have the tree in figure \ref{fig: ex lie tree}
            \begin{figure}[h!]
            \centering

\tikzset{every picture/.style={line width=0.75pt}} 

\begin{tikzpicture}[x=0.75pt,y=0.75pt,yscale=-1,xscale=1]

\draw    (120,60) -- (150,110) ;
\draw [shift={(150,110)}, rotate = 59.04] [color={rgb, 255:red, 0; green, 0; blue, 0 }  ][fill={rgb, 255:red, 0; green, 0; blue, 0 }  ][line width=0.75]      (0, 0) circle [x radius= 3.35, y radius= 3.35]   ;
\draw    (150,60) -- (150,110) ;
\draw [shift={(150,110)}, rotate = 90] [color={rgb, 255:red, 0; green, 0; blue, 0 }  ][fill={rgb, 255:red, 0; green, 0; blue, 0 }  ][line width=0.75]      (0, 0) circle [x radius= 3.35, y radius= 3.35]   ;
\draw    (180,60) -- (150,110) ;
\draw [shift={(150,110)}, rotate = 120.96] [color={rgb, 255:red, 0; green, 0; blue, 0 }  ][fill={rgb, 255:red, 0; green, 0; blue, 0 }  ][line width=0.75]      (0, 0) circle [x radius= 3.35, y radius= 3.35]   ;
\draw    (200,110) -- (230,160) ;
\draw [shift={(230,160)}, rotate = 59.04] [color={rgb, 255:red, 0; green, 0; blue, 0 }  ][fill={rgb, 255:red, 0; green, 0; blue, 0 }  ][line width=0.75]      (0, 0) circle [x radius= 3.35, y radius= 3.35]   ;
\draw    (230,110) -- (230,160) ;
\draw [shift={(230,160)}, rotate = 90] [color={rgb, 255:red, 0; green, 0; blue, 0 }  ][fill={rgb, 255:red, 0; green, 0; blue, 0 }  ][line width=0.75]      (0, 0) circle [x radius= 3.35, y radius= 3.35]   ;
\draw    (260,110) -- (230,160) ;
\draw [shift={(230,160)}, rotate = 120.96] [color={rgb, 255:red, 0; green, 0; blue, 0 }  ][fill={rgb, 255:red, 0; green, 0; blue, 0 }  ][line width=0.75]      (0, 0) circle [x radius= 3.35, y radius= 3.35]   ;
\draw    (120,110) -- (150,160) ;
\draw [shift={(150,160)}, rotate = 59.04] [color={rgb, 255:red, 0; green, 0; blue, 0 }  ][fill={rgb, 255:red, 0; green, 0; blue, 0 }  ][line width=0.75]      (0, 0) circle [x radius= 3.35, y radius= 3.35]   ;
\draw    (150,110) -- (150,160) ;
\draw [shift={(150,160)}, rotate = 90] [color={rgb, 255:red, 0; green, 0; blue, 0 }  ][fill={rgb, 255:red, 0; green, 0; blue, 0 }  ][line width=0.75]      (0, 0) circle [x radius= 3.35, y radius= 3.35]   ;
\draw    (180,110) -- (150,160) ;
\draw [shift={(150,160)}, rotate = 120.96] [color={rgb, 255:red, 0; green, 0; blue, 0 }  ][fill={rgb, 255:red, 0; green, 0; blue, 0 }  ][line width=0.75]      (0, 0) circle [x radius= 3.35, y radius= 3.35]   ;
\draw    (150,160) -- (190,200) ;
\draw [shift={(190,200)}, rotate = 45] [color={rgb, 255:red, 0; green, 0; blue, 0 }  ][fill={rgb, 255:red, 0; green, 0; blue, 0 }  ][line width=0.75]      (0, 0) circle [x radius= 3.35, y radius= 3.35]   ;
\draw    (190,150) -- (190,200) ;
\draw [shift={(190,200)}, rotate = 90] [color={rgb, 255:red, 0; green, 0; blue, 0 }  ][fill={rgb, 255:red, 0; green, 0; blue, 0 }  ][line width=0.75]      (0, 0) circle [x radius= 3.35, y radius= 3.35]   ;
\draw    (230,160) -- (190,200) ;
\draw [shift={(190,200)}, rotate = 135] [color={rgb, 255:red, 0; green, 0; blue, 0 }  ][fill={rgb, 255:red, 0; green, 0; blue, 0 }  ][line width=0.75]      (0, 0) circle [x radius= 3.35, y radius= 3.35]   ;
\draw    (190,200) -- (190,240) ;

\draw (111,92.4) node [anchor=north west][inner sep=0.75pt]    {$1$};
\draw (257,92.4) node [anchor=north west][inner sep=0.75pt]    {$6$};
\draw (187,132.4) node [anchor=north west][inner sep=0.75pt]    {$5$};
\draw (197,92.4) node [anchor=north west][inner sep=0.75pt]    {$9$};
\draw (227,92.4) node [anchor=north west][inner sep=0.75pt]    {$8$};
\draw (177,92.4) node [anchor=north west][inner sep=0.75pt]    {$7$};
\draw (117,42.4) node [anchor=north west][inner sep=0.75pt]    {$3$};
\draw (147,42.4) node [anchor=north west][inner sep=0.75pt]    {$2$};
\draw (177,42.4) node [anchor=north west][inner sep=0.75pt]    {$4$};

\end{tikzpicture}

            \label{fig: ex lie tree}
            \caption{}
            \end{figure}
            then, this can be represented as
            \begin{align*}
                [[1,[3,2,4],7],5,[9,8,6]].
            \end{align*}

        \indent
        \\
        \indent For each element of $Tree_p(r,n)_r$, there is a natural right action of $k[\Sigma_{rn-r+1}]$, by just permuting the inputs of the corresponding tree.  Hence, we can define the $\Sigma$-module $\cU_{n,d}$ with 
        \begin{align*}
            \cU_{n,d}(q) = \begin{cases}
                \uparrow^dk[Tree_p(r,n)_r] & \text{if $q=rn-r+1$, and $r\geq 0$}\\
                0 & \text{otherwise}
            \end{cases}.
        \end{align*}  This has a natural operad structure through grafting of trees and relabeling the inputs accordingly.
        \\
        \indent Define the operadic ideal $J_{n,d}$ generated by the Filippov identity of $n\text{-}Lie_d$ algebras and antisymmetry.  For example, if $n=3$ then the ideal $J_{3,d}$ is generated by all permutations of 

        \begin{centering}

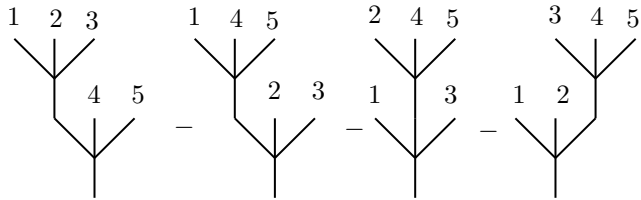
\tikzset{every picture/.style={line width=0.75pt}} 

\begin{tikzpicture}[x=0.75pt,y=0.75pt,yscale=-1,xscale=1]

\draw    (90,110) -- (110,130) ;
\draw    (110,130) -- (110,150) ;
\draw    (110,110) -- (110,130) ;
\draw    (130,110) -- (110,130) ;
\draw    (70,70) -- (90,90) ;
\draw    (90,90) -- (90,110) ;
\draw    (90,70) -- (90,90) ;
\draw    (110,70) -- (90,90) ;
\draw    (250,110) -- (270,130) ;
\draw    (270,130) -- (270,150) ;
\draw    (270,110) -- (270,130) ;
\draw    (290,110) -- (270,130) ;
\draw    (320,110) -- (340,130) ;
\draw    (340,130) -- (340,150) ;
\draw    (340,110) -- (340,130) ;
\draw    (360,110) -- (340,130) ;
\draw    (180,110) -- (200,130) ;
\draw    (200,130) -- (200,150) ;
\draw    (200,110) -- (200,130) ;
\draw    (220,110) -- (200,130) ;
\draw    (160,70) -- (180,90) ;
\draw    (180,90) -- (180,110) ;
\draw    (180,70) -- (180,90) ;
\draw    (200,70) -- (180,90) ;
\draw    (250,70) -- (270,90) ;
\draw    (270,90) -- (270,110) ;
\draw    (270,70) -- (270,90) ;
\draw    (290,70) -- (270,90) ;
\draw    (340,70) -- (360,90) ;
\draw    (360,90) -- (360,110) ;
\draw    (360,70) -- (360,90) ;
\draw    (380,70) -- (360,90) ;

\draw (300,110.4) node [anchor=north west][inner sep=0.75pt]    {$-$};
\draw (233,109.4) node [anchor=north west][inner sep=0.75pt]    {$-$};
\draw (86,54) node [anchor=north west][inner sep=0.75pt]   [align=left] {2};
\draw (65,53) node [anchor=north west][inner sep=0.75pt]   [align=left] {1};
\draw (104,55) node [anchor=north west][inner sep=0.75pt]   [align=left] {3};
\draw (105,92) node [anchor=north west][inner sep=0.75pt]   [align=left] {4};
\draw (127,92) node [anchor=north west][inner sep=0.75pt]   [align=left] {5};
\draw (176,54) node [anchor=north west][inner sep=0.75pt]   [align=left] {4};
\draw (155,53) node [anchor=north west][inner sep=0.75pt]   [align=left] {1};
\draw (194,55) node [anchor=north west][inner sep=0.75pt]   [align=left] {5};
\draw (195,90) node [anchor=north west][inner sep=0.75pt]   [align=left] {2};
\draw (217,90) node [anchor=north west][inner sep=0.75pt]   [align=left] {3};
\draw (266,53) node [anchor=north west][inner sep=0.75pt]   [align=left] {4};
\draw (245,52) node [anchor=north west][inner sep=0.75pt]   [align=left] {2};
\draw (284,54) node [anchor=north west][inner sep=0.75pt]   [align=left] {5};
\draw (245,92) node [anchor=north west][inner sep=0.75pt]   [align=left] {1};
\draw (283,92) node [anchor=north west][inner sep=0.75pt]   [align=left] {3};
\draw (356,53) node [anchor=north west][inner sep=0.75pt]   [align=left] {4};
\draw (335,52) node [anchor=north west][inner sep=0.75pt]   [align=left] {3};
\draw (374,54) node [anchor=north west][inner sep=0.75pt]   [align=left] {5};
\draw (317,92) node [anchor=north west][inner sep=0.75pt]   [align=left] {1};
\draw (339,92) node [anchor=north west][inner sep=0.75pt]   [align=left] {2};
\draw (148,108.4) node [anchor=north west][inner sep=0.75pt]    {$-$};

\end{tikzpicture}

        \end{centering}
         and $Sgn(\sigma)Cor_n - Cor_n^{\sigma}$ for all $\sigma\in \Sigma_n$ for the $n$-Corolla in $\cU_{n,d}(n)$.  We define our $n\text{-}Lie_d$ operad to be $n\text{-}Lie=\cU_{n,d}/J_{n,d}$.

        \subsection{Quadratic Representations for $n\text{-}Lie_d$}\label{sec: quadnlie}
        \indent In this section, we will construct the quadratic representation for the $n\text{-}Lie_d$ operad. Fix $n\geq 2$, $d\in\mathbb{Z}$,  and let $\lambda=(1,2,\dots, 2n-1)$ be the standard $(2n-1)$-cycle and let $\omega=(1,2,\dots, n)$ be the standard $n$-cycle in $\Sigma_{2n-1}$.  
         Let $E_{n,d}$ be a $\Sigma$-module with $E_{n,d}(n)=\uparrow^dk\nu$ is the one-dimensional right $\Sigma_n$-module in degree $d$ with $\nu^{\sigma}=Sgn(\sigma)\nu$ for $\sigma\in \Sigma_n$ and zero everywhere else.  
        \\
        \indent  In this case, the weight $2$ part of the free operad $F(E_{n,d})$ is concentrated in arity $2n-1$ and degree $2d$, with the following description:  
        \begin{align*}
            F(E_{n,d})^{(2)}(2n-1)&=\text{Ind}_{\Sigma_n}^{\Sigma_{2n-1}}(E_{n,d}(n)\otimes E_{n,d}(n))\\
            &= \bigoplus_{\{a_1,\dots, a_{n-1}\}\in\Lambda_{n}^{2n-1}} E_{n,d}(n)\otimes E_{n,d}(n)\{a_1,\dots, a_{n-1}\}
        \end{align*}
        as in equation \ref{eq: Ind description}.  Let $v_{\{a_1,\dots, a_{n-1}\}} = \nu\circ_{\{a_1,\dots, a_{n-1}\}} \nu$ for each $\{a_1,\dots, a_{n-1}\}\in \Lambda_m^{m+n-1}$, $R_{n,d}(2n-1)$ be the right $\Sigma_{2n-1}$-submodule of $F(E)^{(2)}(2n-1)$ in degree $2d$ generated by 
            \begin{align*}
                r_{n,d}=v_{\{n+1,\dots, 2n-1\}} +(-1)^{n}v_{\{1,2,\dots, n-1\}} + (-1)^n\sum_{i=2}^{n-2}v_{\{n-i+1,\dots, n-1,n,1,\dots, n-i-1\}} + (-1)^nv_{\{2,\dots, n\}},
            \end{align*}
            and $R_{n,d}(i)=0$ for all $i\neq 2n-1$. By lemma \ref{relations of n lie}, the relation $r_{n,d}$ gives us exactly the relation for an $n$-Lie algebra of degree $d$.
            \indent By applying every element of $\Sigma_{2n-1}$ to the generator of $R_{n,d}$, then $R_{n,d}$ is spanned by the elements
            \begin{align*}
                r_{n,d}^{\sigma^{-1}}:=Sgn(\sigma)v_{\{\sigma(n+1),\dots, \sigma(2n-1)\}}& + Sgn(\sigma)(-1)^n v_{\{\sigma(1),\dots, \sigma(n-1)\}}\\
                &+ Sgn(\sigma)(-1)^n\sum_{i=2}^{n-2} v_{\{\sigma(n-i+1),\dots, \sigma(n-1),\sigma(n),\sigma(1),\dots,\sigma(n-i-1) \}}\\
                &
                + Sgn(\sigma)(-1)^n v_{\{\sigma(2),\dots, \sigma(n)\}}
            \end{align*}
            for all $\sigma\in \Sigma_{2n-1}$.  This gives us exactly our quadratic representation for $n\text{-}Lie_d\cong F(E_{n,d})/(R_{n,d})$.

    \subsection{General Definition of $n\text{-}Com_d$ Operad}
        \indent For this section, we will assume the standard results about the representation theory of $\Sigma_n$ in characteristic $0$, for more information see \cite{fulton_1996}.
        \\
        \indent The notion of $n$-Com algebras of degree $d$ are a generalization of commutative algebras in the direction $n$-arity operations that are not associative for $n\geq 3$, but associative up to some twisted terms.  Let $n\geq 2$ and fix the partition $\lambda_n=(n,n-1)$ of $2n-1$. Let $T_{\lambda_n}$ be the standard Young tableau on $\lambda_n$ whose entries first increase along the rows and then increase along the columns, see figure \ref{mainYT}.
        
        \begin{figure}[h!]
          \centering  

\tikzset{every picture/.style={line width=0.75pt}} 

\begin{tikzpicture}[x=0.75pt,y=0.75pt,yscale=-1,xscale=1]

\draw    (385.88,90) -- (385.88,140) ;
\draw    (210,190) -- (385.88,190) ;
\draw    (385.88,140) -- (385.88,190) ;
\draw    (210,140) -- (210,190) ;
\draw    (264.12,90) -- (264.12,141.25) -- (264.12,190) ;
\draw    (331.76,90) -- (331.76,190) ;
\draw    (210,90) -- (210,140) ;
\draw    (210,90) -- (385.88,90) ;
\draw    (210,140) -- (440,140) ;
\draw    (385.88,90) -- (440,90) ;
\draw    (440,90) -- (440,140) ;

\draw (282.76,107.4) node [anchor=north west][inner sep=0.75pt]    {$\cdots $};
\draw (284.12,157.4) node [anchor=north west][inner sep=0.75pt]    {$\cdots $};
\draw (221,162.4) node [anchor=north west][inner sep=0.75pt]    {$n+1$};
\draw (407,112.4) node [anchor=north west][inner sep=0.75pt]    {$n$};
\draw (231,112.4) node [anchor=north west][inner sep=0.75pt]    {$1$};
\draw (342,112.4) node [anchor=north west][inner sep=0.75pt]    {$n-1$};
\draw (334,162.4) node [anchor=north west][inner sep=0.75pt]    {$2n-1$};

\end{tikzpicture}
        \caption{The Young tableau $T^{\lambda_n}$}
        \label{mainYT}
        \end{figure}

        \begin{definition}\label{def: n-com}
            \indent   A $n\text{-}Com$ algebra of degree $d$ is a graded $k$-module $C$ with a $n$-ary operation $m_n:C^{\otimes n}\rightarrow C$ of degree $d$ satisfying the following properties.
            \begin{itemize}
                \item For any $n$ elements $u_1,\dots, u_n\in C$, and for $\sigma\in \Sigma_n$, we have
                \begin{align*}
                    m_n(u_1,\dots, u_n)=\xi(\sigma,u_1,\dots, u_n)m_n(u_{\sigma(1)},\dots, u_{\sigma(n)}),
                \end{align*}
                
                \item and we have the following relation
                \begin{align*}
                    \Phi(m_n) = \sum_{\sigma\in C_{T_{\lambda}}} Sgn(\sigma) (m_n\circ_1m_n)^{\sigma}=0.
                \end{align*}
                Explicitly, for any $2n-1$ elements $u_1,\dots, u_{2n-1}\in C$, we have the following relation:
                \begin{align*}
                    \Phi(m_n)(u_1,\dots, u_{2n-1})= \sum_{\sigma\in C_{T_{\lambda}}}Sgn(\sigma)\xi(\sigma,u_1,\dots, u_{2n-1}) ((u_{\sigma^{-1}(1)},\dots, u_{\sigma^{-1}(n)}),u_{\sigma^{-1}(n+1)}, \dots, u_{\sigma^{-1}(2n-1)})=0
                \end{align*}
                where $C_{T_{\lambda}}$ is the subgroup of $\Sigma_{2n-1}$ generated by the permutations that preserve the columns of $T_{\lambda}$, i.e. the subgroup generated by $(1,n+1),\dots, (n-1,2n-1)$.  
            \end{itemize}
        \end{definition}
        \indent The relation for $n\text{-}Com_d$ is independent of the standard Young tableau of type $(n,n-1)$, as they are all permutations of each other.  Furthermore, if the degree $d=0$ and the graded $k$-module is concentrated in degree $0$, then we just call them $n$-Com algebras.
        \indent
        \\
        \begin{example}
        For these examples, we will look at the relation for when $m_n$ is of degree $0$ to simplify the presentation.  For when $m_n$ is of degree $d$, there are extra signs coming from permuting the inputs.
            \begin{itemize}
                \item For $n=2$, these are just the graded commutative $k$-algebras, since $C_{T_{(2,1)}}=\langle (1~3)\rangle$ and this gives us the associative relation
                \begin{align*}
                    m_n(m_n(u_1,u_2),u_3) - m_n(u_1,m_n(u_2,u_3))=0
                \end{align*}
                using commutativity of the multiplication.
                \item For $n=3$, if $C$ is a $3$-Com algebra, then $C_{T_{(3,2)}}$ is generated by $(1~4),(2~5)$ and this gives us the relation
                \begin{align*}
                    m_n(m_n(u_1,u_2,u_3),u_4,u_5) - m_n(u_1,m_n(u_2,u_3,u_4),u_5)+m_n(u_1,u_2,m_n(u_3,u_4,u_5))=m_n(m_n(u_1,u_3,u_5),u_2,u_4).
                \end{align*}
                  
                \item For $n=4$, $C_{T_{(4,3)}}$ is generated by $(1~5),(2~6),(3~7)$ and this gives us the relation
                \begin{align*}
                    &m_n(m_n(u_1,u_2,u_3,u_4),u_5,u_6,u_7) - m_n(u_1,m_n(u_2,u_3,u_4,u_5),u_6,u_7)
                     + m_n(u_1,u_2,m_n(u_3,u_4,u_5,u_6),u_7)\\
                     &- m_n(u_1,u_2,u_3,m_n(u_4,u_5,u_6,u_7))\\
                    &=m_n(m_n(u_1,u_3,u_4,u_6),u_2,u_5,u_7) + m_n(m_n(u_1,u_2,u_4,u_7),u_3,u_5,u_6)-m_n(m_n(u_2,u_4,u_5,u_7),u_1,u_3,u_6)\\
                    &-m_n(m_n(u_1,u_4,u_6,u_7),u_2,u_3,u_5).
                \end{align*}
            \end{itemize}
        \end{example}
        In the examples above, for $n\geq 3$ we have some sort of partial associativity up to some twisted factors, which shows that these are very non-trivial generalization of commutative algebras.    
        \\
        \indent Next, we will give a few examples of $n$-Com algebras arising from commutative algebras. 
        \begin{example}
            \indent Let $A$ be any commutative $k$-algebra and let $D:A\rightarrow A$ be a derivation of $A$ of degree $0$ .  We can define the $n$-arity operation $m_n:A^{\otimes n}\rightarrow A$ as $m_n(a_1,\dots, a_n)=D(a_1\cdots a_n)$.  In general, $m_n$ is not an associative product but it does satisfy the condition to be a $n\text{-}Com$ algebra of degree $0$ for $n\geq 3$ as follows.  For $a_1,\dots, a_{2n-1}\in A$, we have
            \begin{align*}
                \sum_{\sigma\in C_{T_{\lambda_n}}}& Sgn(\sigma)m_n(m_n(a_{\sigma(1)},\dots, a_{\sigma(n-1)},a_n),a_{\sigma(n+1)},\dots, a_{\sigma(2n-1)})\\
                &= \sum_{\sigma\in C_{T_{\lambda_n}}}Sgn(\sigma) D(D(a_{\sigma(1)}\cdots a_{\sigma(n-1)},a_n)a_{\sigma(n+1)}\cdots a_{\sigma(2n-1)})\\
                &=\sum_{\sigma\in C_{T_{\lambda_n}}}\sum_{i=1}^{n-1} Sgn(\sigma) D(a_{\sigma(1)}\cdots D(a_{\sigma(i)})\cdots a_{\sigma(n-1)}a_{n} a_{\sigma(n+1)}\cdots a_{\sigma(2n-1)})\\
                &+ \sum_{\sigma\in C_{T_{\lambda_n}}} Sgn(\sigma) D(a_{\sigma(1)}\cdots a_{\sigma(n-1)}D(a_n)a_{\sigma(n+1)}\cdots a_{\sigma(2n-1)}).
            \end{align*}
            By commutativity, the last sum is just $D(a_1\cdots D(a_n)\cdots a_{2n-1})$ multiplied by a finite alternating sum of $1$'s, with an even amount of elements, which is zero. Let $H_i$ be the subgroup of $C_{T_{\lambda_n}}$ that fixes $i$, which is generated by the transpositions $(1,n+1),\dots, (i-1,n+i-1),(i+1,n+i+1),\dots, (n-1,2n-1)$.  Using commutativity, the first sum becomes
            \begin{align*}
                \sum_{\sigma\in C_{T_{\lambda_n}}}\sum_{i=1}^{n-1}& Sgn(\sigma) D(a_{\sigma(1)}\cdots D(a_{\sigma(i)})\cdots a_{\sigma(n-1)}a_{n} a_{\sigma(n+1)}\cdots a_{\sigma(2n-1)})\\
                &=\sum_{\sigma\in H_i}\sum_{i=1}^{n-1} Sgn(\sigma) D(a_{1}\cdots D(a_{i})\cdots a_{n-1}a_{n} a_{n+1}\cdots a_{2n-1})\\
                &+\sum_{\sigma\in C_{T_{\lambda_n}}\setminus H_i}\sum_{i=1}^{n-1}Sgn(\sigma) D(a_1\cdots a_{n+i-1}D(a_{n+i})a_{n+i+1}\cdots a_{2n-1}).
            \end{align*}
            Since both $H_i$ and $C_{T_{\lambda_n}}$ have even cardinality, then both of these sums are zero.  This shows that $(A,m_n)$ is a $n\text{-}Com$ algebra.
            \indent  To give an explicit example of these type of $n\text{-}Com$ algebras for $n\geq 3$, let $P$ be any Poisson algebra.  Pick any element $\Omega\in P$ and this element constructs a derivation $\{\Omega,-\}$ on $P$ as a commutative algebra.  Then we can define $m_n(f_1,\dots, f_n)=\{\Omega,f_1\cdots f_n\}$, which gives $P$ an additional $n$-Com algebra structure.
        \end{example}
        \indent
        \\
        \indent Next, we will explore some examples of $n$-Com algebras on finite rank modules over a commutative $k$-algebra $A$.  Let $M$ be any $A$-module of finite rank and with basis elements $e_1,\dots, e_m$.  Let $T:M^{\otimes n}\rightarrow M$ be any symmetric $A$-linear map and using the basis elements we can express $T$ as follows:
        \begin{align*}
            T(e_{i_1},\dots, e_{i_n}) = \sum_{i=1}^m\lambda_{i_1,\dots, i_n}^je_j
        \end{align*}
        for some symmetric coefficients $\lambda_{i_1,\dots, i_n}^j\in A$. Plugging these into the defining equation for $n\text{-}Com$, we have
        \begin{align*}
            \sum_{\sigma\in C_{T_{\lambda_n}}}& Sgn(\sigma)T(T(e_{i_{\sigma(1)}},\dots e_{i_{\sigma(n-1)}},e_{i_n}),e_{i_{\sigma(n+1)}},\dots, e_{i_{\sigma(2n-1)}})\\
            &=\sum_{\sigma\in C_{T_{\lambda}}}\sum_{j=1}^{m}Sgn(\sigma)\lambda_{i_{\sigma(1)},\dots, i_{\sigma(n-1)},i_n}^j T(e_j, e_{i_{\sigma(n+1)}},\dots, e_{i_{\sigma(2n-1)}})\\
            &=\sum_{\sigma\in C_{T_{\lambda}}}\sum_{j=1}^{m}\sum_{l=1}^mSgn(\sigma)\lambda_{i_{\sigma(1)},\dots, i_{\sigma(n-1)},i_n}^j\lambda_{j,i_{\sigma(n+1)},\dots, i_{\sigma(2n-1)}}^l e_l
        \end{align*}
        Therefore, for $T$ to be give $M$ a $n$-Com algebra structure, it suffices for us to have
        \begin{align*}
          \sum_{\sigma\in C_{T_{\lambda}}}\sum_{j=1}^{m}Sgn(\sigma)\lambda_{i_{\sigma(1)},\dots, i_{\sigma(n-1)},i_n}^j\lambda_{j,i_{\sigma(n+1)},\dots, i_{\sigma(2n-1)}}^l =0 
        \end{align*}
        for all $1\leq l\leq m$.
        \begin{example}\label{ex: lambda=delta}
            \indent  Let $n\geq 3$, $\delta\in A$, and define $\lambda_{i_1,\dots, i_n}^j=\delta$ whenever $j\in\{i_1,\dots, i_n\}$.  Then we have
                \begin{align*}
                    \sum_{\sigma\in C_{T_{\lambda}}}\sum_{j=1}^{m}&Sgn(\sigma)\lambda_{i_{\sigma(1)},\dots, i_{\sigma(n-1)},i_n}^j\lambda_{j,i_{\sigma(n+1)},\dots, i_{\sigma(2n-1)}}^l\\
                    &=\sum_{\sigma\in C_{T_{\lambda_n}}}\sum_{r=1}^{n-1}Sgn(\sigma)\lambda_{i_{\sigma(1)},\dots, i_{\sigma(n-1)},i_n}^{i_{\sigma(r)}}\lambda_{i_{\sigma(r)},i_{\sigma(n+1)},\dots, i_{\sigma(2n-1)}}^l\\
                    &+\sum_{\sigma\in C_{T_{\lambda_n}}}Sgn(\sigma)\lambda_{i_{\sigma(1)},\dots, i_{\sigma(n-1)},i_n}^{i_{n}}\lambda_{i_{n},i_{\sigma(n+1)},\dots, i_{\sigma(2n-1)}}^l\\
                    &=\sum_{\sigma\in C_{T_{\lambda_n}}}\sum_{r=1}^{n-1}Sgn(\sigma)\delta\lambda_{i_{\sigma(r)},i_{\sigma(n+1)},\dots, i_{\sigma(2n-1)}}^l\\
                    &+\sum_{\sigma\in C_{T_{\lambda_n}}}Sgn(\sigma)\delta\lambda_{i_{n},i_{\sigma(n+1)},\dots, i_{\sigma(2n-1)}}^l.
                \end{align*}
                If $l\notin \{i_1,\dots, i_{2n-1}\}$, then the last two sums are automatically zero by definition.  Otherwise, $l$ could be any number of the elements $i_1,\dots, i_{2n-1}$ (there could be some redundancy in this list of indices).  If $l=i_n$, then the the last sum is always zero since $C_{T_{\lambda_n}}$ is of even order.  On the other hand, suppose $l=i_{s_1}=\cdots =i_{s_t}$ for some $1\leq t\leq 2n-1$, then we can find a subgroup $H_l$ consisting of the permutations that fix any one of the $s_1,\dots, s_t$.  Since $C_{T_{\lambda_n}}$ is even, then subgroup $H_l$ has to be even as well by Lagrange's theorem.  Therefore, the last sum is
                \begin{align*}
                    \sum_{\sigma\in C_{T_{\lambda_n}}} Sgn(\sigma) \delta\lambda^l_{i_n,i_{\sigma(n+1)},\dots, i_{\sigma(2n-1)}}
                    =\sum_{\sigma\in H_l}Sgn(\sigma)\delta^2=0
                \end{align*}
                \\
                \indent By the same argument for the other sum, when we fix $r$, this shows that $A$-module $M$ with the $A$-linear map $T$ is a $n$-Com algebra.   Note that this can fail when $n=2$, as it is not generally associative.
            \end{example}
                
            \begin{example}
                \indent Here is an example derived from geometry.  Let $X=\RR^n$ for $n\geq 1$ and let $Vect(X)$ be the collection of vector fields on $X$. The space $Vect(X)$ is a $C^{\infty}(X)$-module of finite rank with basis $\frac{\partial}{\partial x^i}$ for $i=1,\dots, N$.  Pick a section $\pi\in\Gamma(S^m(T^*M)\otimes TM)$ for $m\geq 3$ expressed as
                \begin{align*}
                    \pi = \sum_{\substack{i_1\leq \cdots\leq i_m } }\sum_{k}\delta_{i_1,\dots, i_m}^k dx^{i_1}\otimes\cdots\otimes dx^{i_m}\otimes \frac{\partial}{\partial x^k}
                \end{align*}
                where 
                \begin{align*}
                    \delta_{i_1,\dots, i_m}^k = \begin{cases}
                        1 & \text{if $k\in\{i_1,\dots, i_m\}$}\\
                        0 & \text{if otherwise}
                    \end{cases}.
                \end{align*}
                Note that $\delta_{i_1,\dots, i_m}^k$ is symmetric in the $i_1,\dots, i_m$ indices.
                We can define a $m$-ary multiplication on $Vect(X)$ through $\pi_m:Vect(X)^{\otimes_{\RR}m}\rightarrow Vect(X)$ defined as
                \begin{align*}
                    \pi_m(V_1,\dots, V_m) = \pi(V_1,\dots, V_m) = \sum_{i_1,\dots, i_m,k} \delta_{i_1,\dots, i_m}^k V_1(x^{i_1})\cdots V_n(x^{i_m}) \frac{\partial}{\partial x^k}.
                \end{align*}
                On the basis elements of $Vect(X)$, we have
                \begin{align*}
                    \pi_m\left(\frac{\partial}{\partial x^{j_1}},\dots, \frac{\partial}{\partial x^{j_m}}\right) &= \sum_{k}\delta_{j_1,\dots, j_m}^k\frac{\partial}{\partial x^k}
                \end{align*}
                This is exactly the case as in example \ref{ex: lambda=delta} with $\delta=1$ and this gives $Vect(X)$ a $n$-Com algebra structure. 
                \\
                \indent Explicitly, if $X=\RR$ and $m\geq 2$, we have
                \begin{align*}
                    \pi_m(V_1,\dots, V_m) = V_1(x)\cdots V_m(x) \frac{d}{dx}
                \end{align*}
                and one can see easily this satisfies the relation for a $n\text{-}Com$ algebra for $m\geq 2$. 
            \end{example}
        \indent
        \\

    \subsection{Quadratic Representation for $n\text{-}Com$}\label{sec: quad n-com}
        \indent Let $H_{n,d}$ be the graded $\Sigma$-module with $H_{n,d}(n)=\uparrow^dk\mu$ where $\mu^{\sigma}=\mu$ and we note $E^{\vee}_{n,d} = H_{n,-d+n-2}$ by definition.  To construct the operad associated to $n$-Com algebras of degree $d$, we will find a certain submodule of $F(H_{n,d}^{\vee})(2n-1)$ that is isomorphic to $\uparrow^{2d}S^{\lambda_n}$, the Specht module with respect to the partition $\lambda_n=(n,n-1)$ of $2n-1$.  This will give us exactly the relations we want for our $n$-Com algebras.
        \\
        \indent For distinct elements $a_1,\dots, a_{n-1}$ in $\{1,\dots, 2n-1\}$, define $u_{\{a_1,\dots, a_{n-1}\}}=\mu\circ_{\{a_1,\dots, a_{n-1}\}}\mu$, which can be thought of as the operation $\mu\circ_1\mu$ with the last $n-1$ inputs correspond to the elements in $a_1,\dots, a_{n-1}$.
        \indent Next, we will describe our relations for $n\text{-}Com_d$ operad, which will be directly connected to the standard Young tableau in figure \ref{mainYT}.  Let $\mathfrak{L}_n=F(H_{n,d})^{(2)}(2n-1)$ and $S_{n,d}$ be a sub $k[\Sigma_{2n-1}]$-module of $\mathfrak{L}_n$ generated by 
        \begin{align*}
           s_{n,d}=\sum_{\sigma\in C_{T_{\lambda}}} Sgn(\sigma) u_{\{n+1,\dots, 2n-1\}}^{\sigma} =\sum_{\sigma\in C_{T_{\lambda}}} Sgn(\sigma) u_{\{\sigma^{-1}(n+1),\dots, \sigma^{-1}(2n-1)\}} 
        \end{align*}
        which is in degree $2d$.  
        \begin{lemma}
            \indent The relation $s_n$ defined in above is exactly the relations for $n$-Com algebras.  
        \end{lemma}
        \begin{proof}
            \indent Let $x_1,\dots, x_{2n-1}$ be formal graded elements in which $\mu$ acts on.  Then we have
            \begin{align*}
                s_{n,d}(x_1,\dots, x_{2n-1}) &= \sum_{\sigma\in C_{T_{\lambda}}} Sgn(\sigma) u^{\sigma}_{[n+1,\dots, 2n-1]}(x_1,\dots, x_{2n-1})\\
                &=\sum_{\sigma\in C_{T_{\lambda}}}Sgn(\sigma)\xi(\sigma,x_1,\dots, x_{2n-1})u_{[n+1,\dots, 2n-1]}(x_{\sigma(1)},\dots, x_{\sigma(2n-1)})\\
                &=\sum_{\sigma\in C_{T_{\lambda}}}Sgn(\sigma)\xi(\sigma,x_1,\dots, x_{2n-1})\mu(\mu(x_{\sigma(1)},\dots, x_{\sigma(n)}),x_{\sigma(n+1)},\dots, x_{\sigma(2n-1)}).
            \end{align*}
             This gets us exactly the relations for $n$-Com algebras. 
        \end{proof}
        \indent Since $S_{n,d}$ gives us exactly the relations for a $n$-Com algebra of degree $d$, we can define the following operad to represent these algebras.

        \begin{definition}
            \indent Define the $n\text{-}Com_d$ operad to be the quadratic operad $\cP(E^{\vee},S_{n,d})$.
        \end{definition}
        \indent 
        \\
        \indent From our intuition so far about $n\text{-}Com_d$ having relations coming from a Standard Young Tableau, we should expect a relation with the irreducible representation corresponding to the Young diagram, in fact we will show $S_{n,d}\cong \uparrow^{2d} S^{\lambda_n}$.  \\
        \indent Let $n\geq 2$ and choose any Young tableau  $T$ of shape $\lambda_n$, for each $(i,j)$th cell of $T$, let $T_{(i,j)}$ represent the number in that cell.  Let $M^{\lambda_n}$ represent the right $k[\Sigma_{2n-1}]$-module generated by all $\lambda_n$-tabloids. Define $\Phi:\uparrow^{2d}M^{\lambda_n}\rightarrow \frak{L}_n$ by taking a $\lambda_n$-tabloid $\uparrow^{2d}\{T\}$ and sending it to 
        \begin{align*}
            \Phi(\uparrow^{2d}\{T\}) &= u_{\{T_{(1,2)},\dots, T_{(n-1,2)}\}},
        \end{align*}
        which is well-defined for any representative $S\in\{T\}$ since they are equivalent up any permutation of the rows in the Young Tableau.
         This gives an isomorphism between $\uparrow^{2d}M^{\lambda_n}$ and $\frak{L}_n$ as right $k[\Sigma_{2n-1}]$-modules through the generators.
        \\
        \indent Next, take $\uparrow^{2d}e_{T_{\lambda_n}}=\sum_{\sigma\in C_{T_{\lambda_n}}}Sgn(\sigma) \uparrow^{2d}\{T_{\lambda_n}\sigma\}$ be the shifted $\lambda_n$-polytabloid corresponding to $T^{\lambda_n}$ as in the standard Young tableau in figure \ref{mainYT} and apply $\Phi$ on to this element to obtain
        \begin{align*}
            \Phi\left(\uparrow^{2d}e_{T_{\lambda_n}}\right) &= \sum_{\sigma\in C_{T_{\lambda_n}}} Sgn(\sigma) \Phi(\uparrow^{2d}\{T_{\lambda_n}\sigma\})\\
            &=\sum_{\sigma\in C_{T_{(n,n-1)}}} Sgn(\sigma) \Phi(\uparrow^{2d}\{T_{\lambda_n}\})^{\sigma}\\
            &=s_{n,d}.
        \end{align*}
         This induces a isomorphism between $S^{\lambda_n}$ and $S_{n,d}$ using the generators and by the hook length formula we have $\text{dim}S_{n,d} = \text{dim}S^{\lambda_n}=C_n$, where $C_n$ is the $n$th Catalan number.  Since we are in characteristic $0$, the space $\mathfrak{L}_n$ decomposes into irreducible representations 
         \begin{align*}
             F(E^{\vee})^{(2)}(2n-1)=\uparrow^{2d}S^{\lambda_n} \oplus \bigoplus_{\lambda_n\lhd\mu} k_{\mu\lambda_n} \uparrow^{2d}S^{\mu},
         \end{align*}
         where $k_{\mu\lambda_n}$ are the Kostka numbers of partitions $\lambda_n$ and $\mu$, and $\lambda\lhd\mu$ whenever $\lambda_1+\cdots+\lambda_i \leq \mu_1+\cdots \mu_i$ for $i\geq 1$.  Taking the quotient we obtain $n\text{-}Com_d(2n-1)\cong \bigoplus_{\lambda_n\lhd\mu} k_{\mu\lambda_n} \uparrow^{2d}S^{\mu}$ for all $n\geq 2$.  The only partitions  $\mu$ of $2n-1$ with $\mu\rhd \lambda_n$ are
         \begin{align*}
             \mu = (n+i, n-i-1)
         \end{align*}
         for $i=1,\dots, n$.  Recall Kostka number $K_{\mu\lambda}$ for partitions $\mu,\lambda$ of $n$ is the number of semistandard Young tableaux of shape $\mu$ and content $\lambda$.  Therefore, the number of semistandard Young tableaux of shape $\mu=(n+i, n-i-1)$ with content $\lambda_n=(n,n-1)$ for $1\leq i\leq n$ is $K_{\mu\lambda_n}=1$ by a simple observation. In conclusion, we have the following description for $(n\text{-}Com_d)(2n-1)$ in terms of the irreducible representations of $\Sigma_{2n-1}$:
         \begin{align*}
             (n\text{-}Com_d)(2n-1) \cong \bigoplus_{1\leq i\leq n} \uparrow^{2d}S^{(n+i,n-i-1)}.
         \end{align*}

    \section{$n\text{-}Lie_d$ and $n\text{-}Com_{-d+n-2}$ are Koszul dual}
        \indent To prove our main theorem that $n\text{-}Lie_d$ and $n\text{-}Com_{-d+n-2}$ are Koszul dual to each other, we need to show $S_{n,-d+n-2}$ and $R_{n,d}^{\perp}$ are equal to each other.  Since we know the dimension of $S_{n,-d+n-2}$ by section \ref{sec: quad n-com}, it suffices to show $S_{n,-d+n-2}\subseteq R_{n,d}^{\perp}$, and then show $R_{n,d}^{\perp}$ also has the same dimension.  The first part is easier in that we can prove this by looking at $S_{n,-d+n-2}$ and $R_{n,d}^{\perp}$ directly and use the non-degenerate equivariant bilinear form.  On the other hand, to show $R_{n,d}^{\perp}$ also has dimension $C_n$, we will need to discuss a sequence of graphs $\{\cO_n\}_{n\geq 2}$, called the odd graphs, for which $R_{n,d}^{\perp}$ is the eigenspace of the eigenvalue $(-1)^{n+1}$ for $\cO_n$ with multiplicity $C_n$.  With these ingredients, we will be able to prove our main result.  

        \subsection{$S_{n,-d+n-2}$ is a subspace of $R_{n,d}^{\perp}$}\label{S_n subset R_nperp}
        \indent For this section, we will show $S_{n,-d+n-2}$ is a linear subspace of $R_{n,d}^{\perp}$, which is the first major part in proving $n\text{-}Com_{-d+n-2}$ is Koszul dual to $n\text{-}Lie_d$.  Recall that we have the subgroup $C_{T^{(n,n-1)}}$ as in definition \ref{def of n lie}, which is generated by the transpositions $(1,n+1),\dots, (n-1,2n-1)$. Any permutation $\sigma\in C_{T^{(n,n-1)}}$ is a product of transpositions $\sigma=(i_1,n+i_1)\cdots (i_k,n+i_k)$ for $i_1<\cdots<i_k$ in $\{1,\dots, n-1\}$ for $1\leq k\leq n$.  Define $\sigma^{\perp}$ to be the product of the transpositions $(j,n+j)$ for $j\in \{1,\dots, n-1\}\setminus \{i_1,\dots, i_k\}$.  If $k=n$, then define $\sigma^{\perp}=id$ and vice versa.  
        Note that $Sgn(\sigma)= (-1)^k$ implies $Sgn(\sigma^{\perp})=(-1)^{n-1+k}$ so $Sgn(\sigma^{\perp})=Sgn(\sigma)(-1)^{n-1}$.  For an example, let $\sigma \in C_{T^{(4,3)}}$ with $\sigma = (1~5)(3~7)$,then $\sigma^{\perp} = (2~6)$ so that $-1 = Sgn(\sigma^{\perp}) = Sgn(\sigma)(-1)^{3} = 1\cdot(-1)^3 = -1.$
        \\
        \indent For making the argument easier to understand, when we are talking about the representation of an element in $F(E)^{(2)}(2n-1)$, say $u_{\{a_1,\dots, a_{n-1}\}}$ or $v_{\{a_1,\dots, a_{n-1}\}}$, we are pertaining to the element $\{a_1,\dots, a_{n-1}\}$, which tells us which right coset this element is in.  These are representing the basis elements of this spaces, and the non-degenerate bilinear form is defined through these basis elements.  
        \\
        \begin{lemma}\label{snsubspace}
            \indent The space $S_{n,-d+n-2}$ is a subspace of $R_{n,d}^{\perp}$. 
        \end{lemma}
        \begin{proof}
            \indent Since $R_{n,d}$ is generated by $r_{n,d}$ as in section \ref{sec: quadnlie}, $S_{n,-d+n-2}$ is generated by $s_{n,-d+n-2}$ as in \ref{sec: quad n-com}, and $\langle -,-\rangle$ is a non-degenerated equivariant bilinear form, then it suffices to show $\langle r_{n,d}^{w^{-1}},s_{n,-d+n-2}\rangle=0$ for all $w\in \Sigma_{2n-1}$.  To help with clarity of the arguments in this proof, we go ahead and restate the relations of $r_{n,d}$ and $s_{n,-d+n-2}$: 
            \begin{align*}
                r_{n,d}^{\omega^{-1}}:=v_{\{\omega(n+1),\dots, \omega(2n-1)\}}& + (-1)^n v_{\{\omega(1),\dots, \omega(n-1)\}}\\
                &+ (-1)^n\sum_{i=2}^{n-2} v_{\{\omega(n-i+1),\dots, \omega(n-1),\omega(n),\omega(1),\dots,\omega(n-i-1) \}}\\
                &
                + (-1)^n v_{\{\omega(2),\dots, \omega(n)\}}
            \end{align*}
            and 
            \begin{align*}
             s_{n,-d+n-2}=\sum_{\sigma\in C_{T_{\lambda}}} Sgn(\sigma) u_{\{\sigma^{-1}(n+1),\dots, \sigma^{-1}(2n-1)\}}.   
            \end{align*}
            \\
            \indent In this proof we will play this game of determining if there are terms in $r_{n,d}^{\omega^{-1}}$ and terms in $s_{n,-d+n-2}$ that have the same representation and only compute the bilinear form on those same elements, since it is zero otherwise.
            \\
            \indent  Let $\omega\in \Sigma_{2n-1}$ and assume there exists $\sigma\in C_{T^{(n,n-1)}}$ and $\tau\in \overline{\Sigma}_{n-1}$ such that $\{\omega(n+1),\dots, \omega(2n-1)\}=\{\sigma\tau(n+1),\dots, \sigma\tau(2n-1)\}$, which does not have $n$ in the representation.  This implies there is only one other term in $r_{n,d}^{\omega^{-1}}$ that does not have a $n$ in the representation by definition of $r_{n,d}$, i.e. there exists $1\leq i\leq n-2$ such that we have $\{\sigma^{\perp}\varphi(n+1),\dots, \sigma^{\perp}\varphi(2n-1)\}$ is equal to one of the following:
            \begin{align*}
                &\{\omega(n-i+1),\dots, \omega(n),\omega(1),\dots,\omega(n-i-1)\}\\
                &\{\omega(1),\dots, \omega(n-1)\}\\
                &\{\omega(2),\dots, \omega(n)\} \\
            \end{align*}
        for some $\varphi\in \overline{\Sigma}_{n-1}$.
            \\
            \indent Since the terms in $s_{n,-d+n-2}$ do not have a $n$ in any of its representations, then we have the following: 
            \begin{align*}
                \langle r_{n,d}^{\omega^{-1}}, s_{n,-d+n-2}\rangle &= \langle Sgn(\omega) (-1)^nv_{\{\omega(n+1),\dots, \omega(2n-1)\}},Sgn(\sigma)u_{\{\sigma\tau(n+1),\dots, \sigma\tau(2n-1)\}}\rangle\\
                &+\langle Sgn(\omega) v_{\{\omega(n-i+1),\dots, \omega(n), \omega(1),\dots, \omega(n-i+1)\}},Sgn(\sigma^{\perp}) u_{\{\sigma^{\perp}\varphi(n+1),\dots, \sigma^{\perp}\varphi(2n-1)\}}\rangle\\
                &=Sgn(\omega)Sgn(\sigma)(-1)^n + Sgn(\omega)Sgn(\sigma^{\perp})\\
                &=Sgn(\omega)\left(Sgn(\sigma)(-1)^n  + Sgn(\sigma^{\perp})\right)\\
                &= Sgn(\omega) \left( Sgn(\sigma)(-1)^n + Sgn(\sigma) (-1)^{n-1}\right)=0.
            \end{align*}
            \\
            \indent Next, lets assume $w\in \Sigma_{2n-1}$ such that $\{\omega(n+1),\dots, \omega(2n-1)\}\neq \{\sigma\tau(n+1),\dots, \sigma\tau(2n-1)\}$ for any $\sigma\in C_{T^{(n,n-1)}}$ and $\tau\in \overline{\Sigma}_{n-1}$. There are two cases where this can happen, either we have $n$ in the representation $\{\omega(n+1),\dots, \omega(2n-1)\}$ or we don't have $n$ in the representation.
            \\
            \indent First, let suppose $\{\omega(n+1),\dots, \omega(2n-1)\}$ does not have a $n$ in its representation, then there is only one other term in $r_{n,d}^{\omega^{-1}}$ that does not have a $n$ in its representation, say $\{\omega(n-i+1),\dots, \omega(n),\omega(1),\dots, \omega(n-i+1)\}$.  This representation can not be equal to $\{\sigma\tau(n+1),\dots, \sigma\tau(2n-1)\}$ for any $\sigma\in C_{T^{(n,n-1)}}$ and $\tau\in \overline{\Sigma}_{n-1}$, since this would imply $\{\omega(n+1),\dots, \omega(2n-1)\}=\{\sigma^{\perp}\varphi(n+1),\dots, \sigma^{\perp}\varphi(2n-1)\}$ for some $\varphi\in \overline{\Sigma}_{n-1}$, which would be a contradiction.  In this case we have $\langle r_{n,d}^{\omega^{-1}},s_{n,-d+n-2}\rangle=0$ since both equations do not have any terms with the same representative.
            \\
            \indent On the other hand, if $n$ is in the representation $\{\omega(n+1),\dots, \omega(2n-1)\}$, then all the other terms in $r_{n,d}^{\omega^{-1}}$ does not have a $n$ in its representation by definition of $r_{n,d}$.  If there are no other terms in $r_{n,d}^{\omega^{-1}}$ that have the same representation as a term in $s_{n,-d+n-2}$, then $\langle r_{n,d}^{\omega^{-1}},s_{n,-d+n-2}\rangle =0$.  Otherwise, suppose there exists $\sigma\in C_{T^{(n,n-1)}}$ and $\tau\in \overline{\Sigma_{n-1}}$ such that $\{\sigma\tau(n+1),\dots, \sigma\tau(2n-1)\}$ is equal to one of the following:
            \begin{align*}
                &\{\omega(1),\dots, \omega(n-1)\}\\
                &\{\omega(n-i+1),\dots, \omega(n), \omega(1),\dots, \omega(n-i-1)\}~~~\text{for some $1\leq i \leq n-2$}\\
                &\{\omega(2),\dots, \omega(n)\}.
            \end{align*}
            For our arguments, let $\{\sigma\tau(n+1),\dots,\sigma\tau(2n-1)\}$ be equal to $\{\omega(n-i+1),\dots, \omega(n), \omega(1),\dots, \omega(n-i-1)\}$ for some $i$ and the other cases are similar.
             If we let
            \begin{align*}
                \sigma=(i_1,n+i_1)\cdots (i_k,n+i_k),
            \end{align*}
            then
            \begin{align*}
                \omega(\{n-i+1,\dots, n\}) \cup \omega(1, \dots, n-i+1\}) = (\{n+1,\dots, 2n-1\}\setminus \{n+i_1,\dots, n+i_k\})\cup \{i_1,\dots, i_k\}
            \end{align*}
            and
            \begin{align*}
                \omega(\{n+1,\dots, 2n-1\}) = \{1,\dots, n\}\setminus \{i_1,\dots, i_k,l\}\cup \{n+i_1,\dots, n+i_k\}
            \end{align*}
            where $l=\omega(n-i)\in \{1,\dots, n-1\}\setminus \{i_1,\dots, i_k\}$.
            \\
            \indent There exists a $t$ with $1\leq t\leq n-i-1$ or $n-i+1\leq t\leq n$ such that $w(t)=n+l$, since $l$ is not in $\{w(n-i+1),\dots, w(n),w(1),\dots, w(n-i-1)\}$, then we must have $n+l$ in it's representation by definition of $s_{n,-d+n-2}$.  We can cycle through $\{w(n-i+1),\dots, w(n),w(1),\dots, w(n-i+1)\}$ till $w(t)$ is not in the representation to get another term of $r_{n,d}^{w^{-1}}$ that has the same representation as $u_{\{\sigma'h(n+1),\dots, \sigma'h(2n-1)\}}$ where $\sigma'=\sigma\cdot (l,n+l)$ for some permutation $h\in \overline{\Sigma}_{n-1}$.  These are the only two terms that will have the same representation as a term in $s_{n,-d+n-2}$.
            \\
            \indent Hence, we obtain
            \begin{align*}
                \langle r_{n,d}^{w^{-1}},s_{n,-d+n-2}\rangle &= Sgn(w) Sgn(\sigma) + Sgn(w)Sgn(\sigma')\\
                &= Sgn(w) (-1)^{k} + Sgn(w) (-1)^{k+1}=0.
            \end{align*}
            \indent This completes the proof.
        \end{proof}
        \subsection{Odd graphs and Catalan numbers}\label{sec: odd and catalan}
            \indent For this section we will state the important definitions and proofs needed to understand main lemma for this section and postpone the more technical work for the appendix to help with the flow of the paper.  
            
            \begin{definition}
                \indent For $n\geq 2$, let $\Lambda(n)$ be the collection of ordered sequences $(a_1,\dots, a_{n-1})$ of elements $a_1,\dots, a_{n-1}\in\{1,\dots, 2n-1\}$.  For each $n\geq 2$, define $\cO_n=G(\Lambda(n))$ to be the associated graph with vertices $V(\cO_n) = \Lambda(n)$ with an edge between $\sigma,\tau\in\Lambda(n)$ if and only if $\sigma\cap\tau =\emptyset$.    
            \end{definition}
            In the literature, these are called the odd graphs $\cO_n=K(2n-1,n-1)$ for $n\geq 2$, where $K(r,s)$ is the Kenser graph, whose spectrum are known in general, \cite{godsil2001algebraic}.

            \begin{example}
                \begin{itemize}
                     \indent For $n=2$, the graph $O_2$ is just the triangle

\begin{figure}[h!]
\centering
\tikzset{every picture/.style={line width=0.75pt}} 

\begin{tikzpicture}[x=0.75pt,y=0.75pt,yscale=-1,xscale=1]

\draw    (280,90) -- (380,190) ;
\draw [shift={(380,190)}, rotate = 45] [color={rgb, 255:red, 0; green, 0; blue, 0 }  ][fill={rgb, 255:red, 0; green, 0; blue, 0 }  ][line width=0.75]      (0, 0) circle [x radius= 3.35, y radius= 3.35]   ;
\draw [shift={(280,90)}, rotate = 45] [color={rgb, 255:red, 0; green, 0; blue, 0 }  ][fill={rgb, 255:red, 0; green, 0; blue, 0 }  ][line width=0.75]      (0, 0) circle [x radius= 3.35, y radius= 3.35]   ;
\draw    (280,90) -- (190,190) ;
\draw [shift={(190,190)}, rotate = 131.99] [color={rgb, 255:red, 0; green, 0; blue, 0 }  ][fill={rgb, 255:red, 0; green, 0; blue, 0 }  ][line width=0.75]      (0, 0) circle [x radius= 3.35, y radius= 3.35]   ;
\draw [shift={(280,90)}, rotate = 131.99] [color={rgb, 255:red, 0; green, 0; blue, 0 }  ][fill={rgb, 255:red, 0; green, 0; blue, 0 }  ][line width=0.75]      (0, 0) circle [x radius= 3.35, y radius= 3.35]   ;
\draw    (380,190) -- (190,190) ;
\draw [shift={(190,190)}, rotate = 180] [color={rgb, 255:red, 0; green, 0; blue, 0 }  ][fill={rgb, 255:red, 0; green, 0; blue, 0 }  ][line width=0.75]      (0, 0) circle [x radius= 3.35, y radius= 3.35]   ;
\draw [shift={(380,190)}, rotate = 180] [color={rgb, 255:red, 0; green, 0; blue, 0 }  ][fill={rgb, 255:red, 0; green, 0; blue, 0 }  ][line width=0.75]      (0, 0) circle [x radius= 3.35, y radius= 3.35]   ;

\end{tikzpicture}
    \caption{The graph $\cO_2$}
    \label{fig: O_2}
\end{figure}
    \item  \indent For $n=3$, $O_3$ is just the Peterson graph

\begin{figure}[h!]
\centering
\tikzset{every picture/.style={line width=0.75pt}} 

\begin{tikzpicture}[x=0.75pt,y=0.75pt,yscale=-1,xscale=1]

\draw    (320,60) -- (300,120) ;
\draw [shift={(300,120)}, rotate = 108.43] [color={rgb, 255:red, 0; green, 0; blue, 0 }  ][fill={rgb, 255:red, 0; green, 0; blue, 0 }  ][line width=0.75]      (0, 0) circle [x radius= 3.35, y radius= 3.35]   ;
\draw [shift={(320,60)}, rotate = 108.43] [color={rgb, 255:red, 0; green, 0; blue, 0 }  ][fill={rgb, 255:red, 0; green, 0; blue, 0 }  ][line width=0.75]      (0, 0) circle [x radius= 3.35, y radius= 3.35]   ;
\draw    (290,80) -- (340,120) ;
\draw [shift={(340,120)}, rotate = 38.66] [color={rgb, 255:red, 0; green, 0; blue, 0 }  ][fill={rgb, 255:red, 0; green, 0; blue, 0 }  ][line width=0.75]      (0, 0) circle [x radius= 3.35, y radius= 3.35]   ;
\draw [shift={(290,80)}, rotate = 38.66] [color={rgb, 255:red, 0; green, 0; blue, 0 }  ][fill={rgb, 255:red, 0; green, 0; blue, 0 }  ][line width=0.75]      (0, 0) circle [x radius= 3.35, y radius= 3.35]   ;
\draw    (350,80) -- (300,120) ;
\draw [shift={(300,120)}, rotate = 141.34] [color={rgb, 255:red, 0; green, 0; blue, 0 }  ][fill={rgb, 255:red, 0; green, 0; blue, 0 }  ][line width=0.75]      (0, 0) circle [x radius= 3.35, y radius= 3.35]   ;
\draw [shift={(350,80)}, rotate = 141.34] [color={rgb, 255:red, 0; green, 0; blue, 0 }  ][fill={rgb, 255:red, 0; green, 0; blue, 0 }  ][line width=0.75]      (0, 0) circle [x radius= 3.35, y radius= 3.35]   ;
\draw    (320,60) -- (340,120) ;
\draw [shift={(340,120)}, rotate = 71.57] [color={rgb, 255:red, 0; green, 0; blue, 0 }  ][fill={rgb, 255:red, 0; green, 0; blue, 0 }  ][line width=0.75]      (0, 0) circle [x radius= 3.35, y radius= 3.35]   ;
\draw [shift={(320,60)}, rotate = 71.57] [color={rgb, 255:red, 0; green, 0; blue, 0 }  ][fill={rgb, 255:red, 0; green, 0; blue, 0 }  ][line width=0.75]      (0, 0) circle [x radius= 3.35, y radius= 3.35]   ;
\draw    (350,80) -- (290,80) ;
\draw [shift={(290,80)}, rotate = 180] [color={rgb, 255:red, 0; green, 0; blue, 0 }  ][fill={rgb, 255:red, 0; green, 0; blue, 0 }  ][line width=0.75]      (0, 0) circle [x radius= 3.35, y radius= 3.35]   ;
\draw [shift={(350,80)}, rotate = 180] [color={rgb, 255:red, 0; green, 0; blue, 0 }  ][fill={rgb, 255:red, 0; green, 0; blue, 0 }  ][line width=0.75]      (0, 0) circle [x radius= 3.35, y radius= 3.35]   ;
\draw    (320,20) -- (320,60) ;
\draw [shift={(320,60)}, rotate = 90] [color={rgb, 255:red, 0; green, 0; blue, 0 }  ][fill={rgb, 255:red, 0; green, 0; blue, 0 }  ][line width=0.75]      (0, 0) circle [x radius= 3.35, y radius= 3.35]   ;
\draw [shift={(320,20)}, rotate = 90] [color={rgb, 255:red, 0; green, 0; blue, 0 }  ][fill={rgb, 255:red, 0; green, 0; blue, 0 }  ][line width=0.75]      (0, 0) circle [x radius= 3.35, y radius= 3.35]   ;
\draw    (390,60) -- (350,80) ;
\draw [shift={(350,80)}, rotate = 153.43] [color={rgb, 255:red, 0; green, 0; blue, 0 }  ][fill={rgb, 255:red, 0; green, 0; blue, 0 }  ][line width=0.75]      (0, 0) circle [x radius= 3.35, y radius= 3.35]   ;
\draw [shift={(390,60)}, rotate = 153.43] [color={rgb, 255:red, 0; green, 0; blue, 0 }  ][fill={rgb, 255:red, 0; green, 0; blue, 0 }  ][line width=0.75]      (0, 0) circle [x radius= 3.35, y radius= 3.35]   ;
\draw    (360,150) -- (340,120) ;
\draw [shift={(340,120)}, rotate = 236.31] [color={rgb, 255:red, 0; green, 0; blue, 0 }  ][fill={rgb, 255:red, 0; green, 0; blue, 0 }  ][line width=0.75]      (0, 0) circle [x radius= 3.35, y radius= 3.35]   ;
\draw [shift={(360,150)}, rotate = 236.31] [color={rgb, 255:red, 0; green, 0; blue, 0 }  ][fill={rgb, 255:red, 0; green, 0; blue, 0 }  ][line width=0.75]      (0, 0) circle [x radius= 3.35, y radius= 3.35]   ;
\draw    (250,60) -- (290,80) ;
\draw [shift={(290,80)}, rotate = 26.57] [color={rgb, 255:red, 0; green, 0; blue, 0 }  ][fill={rgb, 255:red, 0; green, 0; blue, 0 }  ][line width=0.75]      (0, 0) circle [x radius= 3.35, y radius= 3.35]   ;
\draw [shift={(250,60)}, rotate = 26.57] [color={rgb, 255:red, 0; green, 0; blue, 0 }  ][fill={rgb, 255:red, 0; green, 0; blue, 0 }  ][line width=0.75]      (0, 0) circle [x radius= 3.35, y radius= 3.35]   ;
\draw    (280,150) -- (300,120) ;
\draw [shift={(300,120)}, rotate = 303.69] [color={rgb, 255:red, 0; green, 0; blue, 0 }  ][fill={rgb, 255:red, 0; green, 0; blue, 0 }  ][line width=0.75]      (0, 0) circle [x radius= 3.35, y radius= 3.35]   ;
\draw [shift={(280,150)}, rotate = 303.69] [color={rgb, 255:red, 0; green, 0; blue, 0 }  ][fill={rgb, 255:red, 0; green, 0; blue, 0 }  ][line width=0.75]      (0, 0) circle [x radius= 3.35, y radius= 3.35]   ;
\draw    (250,60) -- (280,150) ;
\draw [shift={(280,150)}, rotate = 71.57] [color={rgb, 255:red, 0; green, 0; blue, 0 }  ][fill={rgb, 255:red, 0; green, 0; blue, 0 }  ][line width=0.75]      (0, 0) circle [x radius= 3.35, y radius= 3.35]   ;
\draw [shift={(250,60)}, rotate = 71.57] [color={rgb, 255:red, 0; green, 0; blue, 0 }  ][fill={rgb, 255:red, 0; green, 0; blue, 0 }  ][line width=0.75]      (0, 0) circle [x radius= 3.35, y radius= 3.35]   ;
\draw    (390,60) -- (360,150) ;
\draw [shift={(360,150)}, rotate = 108.43] [color={rgb, 255:red, 0; green, 0; blue, 0 }  ][fill={rgb, 255:red, 0; green, 0; blue, 0 }  ][line width=0.75]      (0, 0) circle [x radius= 3.35, y radius= 3.35]   ;
\draw [shift={(390,60)}, rotate = 108.43] [color={rgb, 255:red, 0; green, 0; blue, 0 }  ][fill={rgb, 255:red, 0; green, 0; blue, 0 }  ][line width=0.75]      (0, 0) circle [x radius= 3.35, y radius= 3.35]   ;
\draw    (320,20) -- (390,60) ;
\draw [shift={(390,60)}, rotate = 29.74] [color={rgb, 255:red, 0; green, 0; blue, 0 }  ][fill={rgb, 255:red, 0; green, 0; blue, 0 }  ][line width=0.75]      (0, 0) circle [x radius= 3.35, y radius= 3.35]   ;
\draw [shift={(320,20)}, rotate = 29.74] [color={rgb, 255:red, 0; green, 0; blue, 0 }  ][fill={rgb, 255:red, 0; green, 0; blue, 0 }  ][line width=0.75]      (0, 0) circle [x radius= 3.35, y radius= 3.35]   ;
\draw    (320,20) -- (250,60) ;
\draw [shift={(250,60)}, rotate = 150.26] [color={rgb, 255:red, 0; green, 0; blue, 0 }  ][fill={rgb, 255:red, 0; green, 0; blue, 0 }  ][line width=0.75]      (0, 0) circle [x radius= 3.35, y radius= 3.35]   ;
\draw [shift={(320,20)}, rotate = 150.26] [color={rgb, 255:red, 0; green, 0; blue, 0 }  ][fill={rgb, 255:red, 0; green, 0; blue, 0 }  ][line width=0.75]      (0, 0) circle [x radius= 3.35, y radius= 3.35]   ;
\draw    (360,150) -- (280,150) ;
\draw [shift={(280,150)}, rotate = 180] [color={rgb, 255:red, 0; green, 0; blue, 0 }  ][fill={rgb, 255:red, 0; green, 0; blue, 0 }  ][line width=0.75]      (0, 0) circle [x radius= 3.35, y radius= 3.35]   ;
\draw [shift={(360,150)}, rotate = 180] [color={rgb, 255:red, 0; green, 0; blue, 0 }  ][fill={rgb, 255:red, 0; green, 0; blue, 0 }  ][line width=0.75]      (0, 0) circle [x radius= 3.35, y radius= 3.35]   ;

\end{tikzpicture} 
\caption{The Peterson Graph}
\label{fig: peterson graph}
\end{figure}
                \end{itemize}
            \end{example}
            \indent
            \\
            \begin{definition}\label{LinearMapRep}
                \indent For $n\geq 2$, the adjacency matrix $B(O_n)$ associated with $O_n$ is defined as
                \begin{align*}
                    B(O_n)_{\tau,\omega}=\begin{cases}
                        1 & \text{if $\tau\cap\omega=\emptyset$}\\
                        0 & \text{otherwise}
                    \end{cases}
                \end{align*}
                for any $\omega,\tau\in\Lambda(n)$ with respect to the ordered basis on $k\Lambda(n)$.  The associated linear map $T_n:k\Lambda(n)\rightarrow k\Lambda(n)$ is defined as
                \begin{align*}
                    T_n(\omega) = \sum_{\substack{\tau\in\Lambda(n)\\ \tau\cap \omega=\emptyset}} \tau.
                \end{align*}
            \end{definition}

        \indent Recall that if we have a graph $G$ with eigenvalues $\lambda_n\geq \cdots\geq \lambda_1$, then its spectrum is
        \begin{align*}
            Spec(G) = \begin{pmatrix}
                \lambda_n &\cdots &\lambda_1\\
                m_n & \cdots & m_1
            \end{pmatrix}
        \end{align*}
        where $m_i$ is the multiplicity of $\lambda_i$.
        \indent We can compute the spectrum of $O_2$ and $O_3$ fairly easily using basic linear algebra tools.

        \begin{example}
            \begin{itemize}
                \item \indent For $O_2$, its adjancency matrix is
            \begin{align*}
                \begin{pmatrix}
                    0 & 1 & 1\\
                    1 & 0 & 1\\
                    1 & 1 & 0
                \end{pmatrix}
            \end{align*}
            which it is easy to show that its spectrum is
            \begin{align*}
                Spec(O_2)= \begin{pmatrix}
                    2 & -1\\
                    1 & 2
                \end{pmatrix}
            \end{align*}
                \item \indent For $O_3$, this is the Peterson graph, and it is known in the literature to have spectrum
            \begin{align*}
               Spec(O_3)= \begin{pmatrix}
                    3 & 1 & -2\\
                    1 & 5 & 4
                \end{pmatrix}
            \end{align*}
            \end{itemize}
        \end{example}

        \indent Before we compute the spectrum of $O_n$ for $n\geq 2$, we need to state a few results about a particular proper Riordan array that is associated with the spectrum of $O_n$.  A Riordan array is a pair $(d(t),h(t))$ of formal power-series such that $d(0)\neq 0$ and $h(0)\neq 0$ as defined in \cite{Riordan}.  This defines a infinite, lower triangular array $\{d_{n,k}\}_{n,k\in\mathds{N}}$ such that
        \begin{align*}
            d_{n,k} = [t^n] d(t) (th(t))^k
        \end{align*}
        where $[t^n]$ is saying to extract the coefficient in front of $t^n$.  In other words, $d(t)(th(t))^k$ is the generating function for the $k$th column of this array.

            \indent Next, we will define a proper Riordan array that was first introduced by B. Shapiro in \cite{Catalan} in the study of a walk problem on the non-negative quadrant on the integral square lattice in two-dimensional Euclidean space, which they called the Catalan triangle.
        \begin{definition}
            \indent Let $\cE=\{\cE_{n,k}\}_{n,k\in\mathds{N}}$ be the lower triangular array with 
            \begin{align*}
                \cE_{n,k} = \cE_{n-1,k-1} + 2\cE_{n-1,k} + \cE_{n-1,k+1}
            \end{align*}
            for all $n,k\geq 0$ and define
            \begin{align*}
                \cE_{0,0}&=1\\
                \cE_{n,k}& = 0 ~~~\text{for $n<k$}.
            \end{align*}
        \end{definition}
        \indent One can use the explicit characterizations in \cite{Riordan} to show that $\cE$ is a proper Riordan array with 
        \begin{align*}
            d(t) = h(t) = \frac{c(t) -1}{t}
        \end{align*}
        where $c(t)$ is the generating function for the Catalan numbers.  In other words, we have $\cE_{n,0}=C_{n+1}$, the $(n+1)$th Catalan number.  For more information and connections for the Catalan numbers  and its plethora of connections to different objects in mathematics, see \cite{stanley2015catalan}.  The first few lines of this lower triangular array are
            \begin{align*}
			\begin{matrix}
			1 \\
			2 & 1\\
			5 & 4 & 1\\
			14 & 14 & 6 & 1\\
			42 & 48 & 27 & 8 & 1\\
			132 & 165 & 110 & 44 & 10 & 1\\
			429 & 572 & 429 & 208 & 65 & 12 & 1\\
			\vdots & \vdots & \vdots & \vdots & \vdots & \vdots & \vdots\\
			\end{matrix}
			\end{align*}

    Using the proper Riordan array $\cE$, we can find the spectrum for $\cO_n$ for general $n\geq 2$ by using a useful property about the sequence of adjacency matrices $\{B(\cO_n)\}_{n\geq 2}$ in which we can use induction.  For the proof see the appendix \ref{proof of spec}.
    \begin{lemma}\label{Spec of O_n}
            \indent For any $n\geq 2$, $O_n$ has eigenvalues $\lambda_1^n,\dots, \lambda_n^n$ where
            \begin{align*}
               \lambda_i^n = (-1)^{n+i}i.
            \end{align*}
            If $m_{n,i}$ are the multiplicities  for $\lambda_i^n$, then 
            \begin{align*}
                m_{n,i} = \cE_{n-1,i-1}.
            \end{align*}
        \end{lemma}
        This lemma shows that the dimension of $\text{ker}(B(\cO_n)-\lambda_1^nI)$ is $C_n$, which is a crucial part in proving our main result, since this is connected with the space $R_{n,d}^{\perp}$.

        \subsection{Main Theorem}    
        \indent We have essentially all the main ingredients to prove the main theorem.  The final part is putting everything together and connecting $R_{n,d}^{\perp}$ with the eigenspaces of the odd graphs as in \ref{sec: odd and catalan}.  Recall from section \ref{sec: koszul dual}, we have a non-degenerate $\Sigma_{2n-1}$-invariant bilinear form $\langle -,-\rangle : F(E_{n,d}^{\vee})^{(2)}(2n-1)\otimes F(E_{n,d})^{(2)}(2n-1)\rightarrow k$,  which gives us the following description for $R_{n,d}^{\perp}$. 

             \begin{lemma}\label{description of R^{perp}}
                If $$U=\sum_{\{a_1,\dots, a_{n-1}\}\in\Lambda_n^{2n-1}}\Gamma_{\{a_1,\dots, a_{n-1}\}} u_{\{a_1,\dots, a_{n-1}\}}\in R_{n,d}^{\perp},$$ then the coefficients satisfy the following equations.
                \begin{itemize}
                    \item For each $\{a_1,\dots, a_{n-1}\}\in\Lambda_n^{2n-1}$ and $\tau\in \Sigma_{n-1}$ we have $\Gamma_{\{a_{\tau(1)},\dots, a_{\tau(n-1)}\}} = \Gamma_{\{a_1,\dots, a_{n-1}\}}$.
                    \item For any $\sigma\in Sh(n,n-1)$ we have the equations:
                    \begin{align*}
                        \Gamma_{\{\sigma(n+1),\dots, \sigma(2n-1)\}} = -(-1)^n\sum_{i=1}^n\Gamma_{\{\sigma(1),\dots, \widehat{\sigma(i)},\dots, \sigma(n)\}},
                    \end{align*}
                    where $\widehat{\sigma(i)}$ means we are taking it out of the sequence.
                \end{itemize}
                With these two equations, $U$ is expressed as
                \begin{align*}
                    U=\sum_{\sigma\in Sh(n,n-1)} \Gamma_{\{\sigma(n+1),\dots, \sigma(2n-1)\}}\left(\sum_{\tau\in \overline{\Sigma_{n-1}}} u_{\{\sigma\tau(n+1),\dots, \sigma\tau(2n-1)\}}\right).
                \end{align*}

            \end{lemma}
            \begin{proof}
                   \indent The proof for this lemma comes from explicitly writing down the formulas for an arbitrary element $U$ to be in $R_{n,d}^{\perp}$. 
            \end{proof}
           \indent 
           \\
           \indent Recall that we have the right $\Sigma_{2n-1}$-module $\mathfrak{L}_n$ is spanned by the elements $u_{\{a_1,\dots, a_{n-1}\}}$ for $a_1<\cdots<a_{n-1}$ in \ref{sec: quad n-com}.  There is a natural isomorphism $\mathfrak{L}_n\rightarrow k\Lambda(n)$ by sending $\overline{u}_{\{a_1,\dots, a_{n-1}\}}=\sum_{\tau\in\Sigma_{n-1}}u_{a_{\tau(1)},\dots, a_{\tau(n-1)}}$ to the ordered sequence $(a_1,\dots, a_{n-1})$. Through this isomorphism, $R_{n,d}^{\perp}\cong \uparrow^{2(-d+n-2)}\mathfrak{R}_{n}$, where $\mathfrak{R}_{n}$ is the subspace consisting of $u=\sum_{\omega\in \Lambda(n)} \Gamma_{\omega} \omega\in k\Lambda(n)$
           such that 
           \begin{align*}
               (-1)^n\Gamma_{\omega} +\sum_{\substack{\tau\in \Lambda(n)\\ \tau\cap \omega=\emptyset}} \Gamma_{\tau}=0,
           \end{align*}
           which is exactly the same equation as in lemma \ref{description of R^{perp}}. 

           \begin{lemma}
               \indent The space $\mathfrak{R}_n$ is equal to $\text{ker}(B(\cO_n)-\lambda_1^nI)$.
           \end{lemma}
           \begin{proof}
               \indent Suppose $u=\sum_{\omega\in\Lambda(n)}\Gamma_{\omega}\omega\in \text{ker}(B(\cO_n) - \lambda_1^nI)$, then by the linear map representation $T_n$ as in definition \ref{LinearMapRep} we have
               \begin{align*}
                  0=(T_n - \lambda_1^nI)(u)& = \sum_{\omega\in\Lambda(n)}\Gamma_{\omega}(T_n(\omega) - \lambda_1^n\omega)\\
                  &=\sum_{\omega\in\Lambda(n)}\Gamma_{\omega}(\sum_{\substack{\tau\in\Lambda(n)\\ \tau\cap \omega=\emptyset}}\tau - \lambda_1^n\omega)\\
                  &=\sum_{\substack{\omega,\tau\in \Lambda(n)\\\tau\cap\omega=\emptyset}}\Gamma_{\omega}\tau - \lambda_1^n\sum_{\omega\in\Lambda(n)}\Gamma_{\omega}\omega\\
                  &=\sum_{\substack{\omega,\tau\in \Lambda(n)\\\tau\cap\omega=\emptyset}}\Gamma_{\tau}\omega - \lambda_1^n\sum_{\omega\in\Lambda(n)}\Gamma_{\omega}\omega\\
                  &=\sum_{\omega\in\Lambda(n)}(\sum_{\substack{\tau\in\Lambda(n)\\\tau\cap\omega=\emptyset}}\Gamma_{\tau} - \lambda_1^n\Gamma_{\omega})\omega
               \end{align*}
               where we switched $\tau$ and $\omega$ the second the last line above.
               The above equation implies 
               \begin{align*}
                   0&=\sum_{\substack{\tau\in\Lambda(n)\\\tau\cap\omega=\emptyset}}\Gamma_{\tau} - \lambda_1^n\Gamma_{\omega}\\
                   &=\sum_{\substack{\tau\in\Lambda(n)\\\tau\cap\omega=\emptyset}}\Gamma_{\tau} +(-1)^n\Gamma_{\omega}\\
               \end{align*}
               which is exactly the condition on the coefficients for the elements in $\mathfrak{R}_{n,d}$.  This proves that these two spaces are equal to each other.  
           \end{proof}
            \indent Since $S_{n,-d+n-2}$ is a subspace of $R_{n,d}^{\perp}$ by lemma \ref{snsubspace} and now both $R_{n,d}^{\perp}$ and $S_{n,-d+n-2}$ have dimension $C_n$ by lemma \ref{Spec of O_n} and section \ref{sec: quad n-com}, respectively, then this shows they are equal.  This finally gives us the following theorem.
\\
            \begin{theorem}\label{maintheorem}
                \indent The operads $n\text{-}Lie_d$ and $n\text{-}Com_{-d+n-2}$ are Koszul dual.
            \end{theorem}
            \indent
            \\
            \indent As a consequence of this, we have a $k[\Sigma_{2n-1}]$-module homomorphism $n\text{-}Lie_d(2n-1)\cong S_{n,d}\cong \uparrow^{2d}S^{(n,n-1)}$.  Note that this theorem will fail in positive characteristic as we will need more relations for $n\text{-}Com_{-d+n-2}$ to make sure it is Koszul dual to $n\text{-}Lie_d$.  We are unsure what the extra relations would be in the case when the characteristic is positive. 
            \indent On the case when these operads are Koszul, this is a open question, and the context of the next paper.  We will explore the operad $n-Com_d$ more closely to determine any nice properties we can obtain through the combinatorics of Young tableaux.

    \appendix

        \section*{Appendix}
            \indent For the appendix, we give the proofs needed to show $R_{n,d}^{\perp}$ has dimension $C_n$. Section A is for defining a sequence of matrices, called a balanced matrix sequence, which has the property that the characteristic polynomials for these matrices can be defined in a recursive manner.  This is useful for finding the multiplicities for a sequence of matrices in a induction argument. Section B is used to show that the sequence of matrices one obtained from the Odd graphs constructs a balanced matrix sequence, which is used in the section C to find the multiplicities of the eigenvalues $(-1)^{n+1}$ for the matrices $B(O_n)$ is $C_n$.
    
        \section{Balanced Matrix Sequences}
         \indent For this section, we will look at a sequence of block matrices with certain properties and show that we can relate the characteristic polynomial between them, so we can apply this to the adjacency matrix of our sequence of graphs $O_n$.  

        \begin{definition}\label{balancedmatrix}
            \indent A balanced matrix sequence is a triple $\{(B(n),C(n),D(n))\}_{n\geq 1}$ of matrices with 
            \begin{align*}
                B(n)&=\begin{pmatrix}
                    0 & C(n)\\
                    C(n)^T & J
                \end{pmatrix}\\\\
                C(n)&= \begin{pmatrix}
                    0 & D(n)\\
                    B(n-1) & J
                \end{pmatrix}
            \end{align*}
            with
            with $$D(n)^TD(n) + I =JB(n-1)^2J$$, 
            where $J$ is the anti-diagonal square identity matrix. We will denote such sequences as $(B,C,D)$.  
            \end{definition}
             For any such sequence, define a sequence of tuples $(N_n,M_n)$ where $C(n)$ is a $N_n\times M_n$ matrix and all the other sizes of the matrices are derived from these two indices. For any square matrix $A$, denote by $P_A(x)=det(A-xI)$ the characteristic polynomial of $A$.
             \\
        \indent
        \begin{theorem}\label{Balanced Theorem}
            If $(B,C,D)$ is a balanced matrix sequence,
            then we have the following relationship between the characteristic polynomials:
            \begin{align*}
                P_{B(n)}(x) = -(-1)^{N_n+N_{n-1}}x^{N_n-M_n}P_{B(n-1)}(-x)^2 P_{B(n-1)}(x-1)P_{B(n-1)}(x+1)
            \end{align*}
        \end{theorem}
        \begin{proof}
            \indent This proof just uses some standard theorems from linear algebra.  First, note by standard multiplication of block matrices we have
            \begin{align*}
                C(n)^TC(n) &= \begin{pmatrix}
                    B(n-1)^2 & B(n-1)J\\
                    JB(n-1) & D(n)^TD(n)+I
                \end{pmatrix}\\
                &=\begin{pmatrix}
                    B(n-1)^2 & B(n-1)J\\
                    JB(n-1) & JB(n-1)^2J
                \end{pmatrix}\\
            \end{align*}
            \indent By definition of characteristic polynomial
            \begin{align*}
                P_{B(n)}(\lambda)&=det(B(n)-\lambda I)\\
                &=det \begin{pmatrix}
                     -\lambda I & C(n)\\
                     C(n)^T & J - \lambda I
                \end{pmatrix}\\
                &=det(-\lambda I) det(J-\lambda I + \frac{1}{\lambda}C(n)^TC(n))\\
                &=(-1)^{N_n}\lambda^{N_n-M_n} det(\lambda J - \lambda^2 I + C(n)^TC(n)),
            \end{align*}
            where we used Schur's complement for the second determinant. 
            Row-reducing the matrix $\lambda J-\lambda^2 I+C(n)^TC(n)$, we obtain the matrix 
            \begin{align*}
                 \begin{pmatrix}
                    0 & -f(\lambda,B(n-1))J\\
                    JB(n-1) +\lambda J & JB(n-1)^2J - \lambda^2J
                \end{pmatrix}   
            \end{align*}
            where 
            \begin{align*}
                f(x,b) = x^3 - x^2b - x(1+b^2)- b +b^3
            \end{align*}
            Note that $f(x,b)$ can be decomposed as
            \begin{align*}
                f(x,b) = (b-x -1)(b-x+1)(b+x)
            \end{align*}
            so that our matrix becomes
            \begin{align*}
                \begin{pmatrix}
                    0 & -(B(n-1)-(\lambda+1)I)(B(n-1)-(\lambda-1)I)(B(n-1)+\lambda I)J\\
                    JB(n-1)+\lambda J & JB(n-1)^2J - \lambda^2I
                \end{pmatrix}.
            \end{align*}
            \indent Computing the determinate of this gets us our result.  
        \end{proof}
    
    \section{The Odd graphs $\cO_n$}\label{sec: odd graphs proofs}
        \indent Here, we will give the details of lemma \ref{Spec of O_n}, by showing that the adjacency matrices of the Odd graphs give us a Balanced matrix sequence to use a inductive argument. The technical details here are just to help with showing that the adjacency matrices has the block matrix form we want.
        \subsection{The Vertices of $\cO_n$}
        \indent Recall the definition of $\Lambda(n)$ from section \ref{sec: odd and catalan}, which give us the basis elements that construct our adjacency matrix $B(\cO_n)$.  For each $1\leq i\leq n+1$, define $\Lambda_i(n)$ to be the subset of $\Lambda(n)$ consisting of elements starting with $i$ and $\Lambda_{i+}(n)$ to be the set $\prod_{j=i+1}^{n+1}\Lambda_i(n)$.
                 Then for each $1\leq i\leq n$, define $\Lambda_{i,i+1}(n)$ to be the subset of $\Lambda_i(n)$ consisting of ordered lists containing both $i$ and $i+1$, and  $\Lambda_{i,i+1,+}$ to be the complement of $\Lambda_{i,i+1}(n)$ in $\Lambda_i(n)$. 

        \indent Define the usual set operations on the elements of $\Lambda(n)$ by using the associated set to each sequence.  Furthermore, we can put a natural lexographic ordering on $\Lambda(n)$, which gives us a way to define a distance function between any two elements.  For $\sigma,\tau\in \Lambda(n)$, define $d(\sigma,\tau)$ to be equal to one plus the number of elements between them with the ordering and $0$ if they are the same element.
            \\
            \indent One particular case that will be useful for us is the following.  Let $\omega=(a_1,\dots, a_{n-1})\in \Lambda(n)$ which is not $(n+1,\dots, 2n-1)$, and suppose $\tau\in\Lambda(n)$ such that $\omega\leq \tau$ and $d(\omega,\tau)=1$.  In $\omega$, there is a maximal $i$ with $1\leq i\leq n-1$ such that
            \begin{align*}
                \omega=(a_1,\dots, a_i, n+1+i,n+2+i,\dots, 2n-1)
            \end{align*}
            with $a_i<n+i$.  Then $\tau$ can be explicitly described as 
            \begin{align*}
                \tau=(a_1,\dots, a_{i-1},a_i+1,a_i+2,\dots, a_{i}+n-i).
            \end{align*}
            \indent For our purposes, we want to preserve the ordering and the distance at the same time.  Define the signed distance function $d_S$ with
            \begin{align*}
                d_S(\sigma,\tau) = \begin{cases}
                    d(\sigma,\tau) & \text{if $\sigma\leq \tau$}\\
                    -d(\sigma,\tau) & \text{if $\tau\leq \sigma$}
                \end{cases}.
            \end{align*}
            It is a easy consequence of this definition that we have the following identity for any $\sigma,\tau,\omega\in\Lambda(n)$:
            \begin{align*}
                d_S(\sigma,\tau) + d_S(\tau,\omega) = d_S(\sigma,\omega).
            \end{align*}
            \\
            \indent Next, we have a few technical lemmas that describes the relationship between the subsets and some properties of the signed distance function.

            \begin{lemma}\label{Prop of Lambda(n)}
                \indent We have the following properties for $\Lambda(n)$.
                \begin{itemize}
                    \item[(a)]  We have the following cardinalities:
                    \begin{align*}
                            |\Lambda(n)|&=\binom{2n-1}{n-1}
                            &|\Lambda_i(n)|=\binom{2n-1-i}{n-2}\\ 
                            |\Lambda_{i,i+1}(n)|&= \binom{2n-2-i}{n-3}
                            &|\Lambda_{i,i+1,+}(n)| = \binom{2n-2-i}{n-2}\\
                            |\Lambda_{i+}(n)|&= \binom{2n-1-i}{n-1}
                        \end{align*}
                    \item[(b)] We have $|\Lambda_{2}(n)|=|\Lambda_{2+}(n)|=\frac{1}{2}|\Lambda_{1+}(n)|$.
                    \item[(c)] For each $\sigma\in\Lambda(n)$, the set of elements disjoint from $\sigma$ has cardinality $n$.
                    \item[(d)] For each $\sigma\in\Lambda_2(n)$, there is a unique $\Gamma_2(\sigma)\in\Lambda_{2+}(n)$ such that $\Gamma_2(\sigma)\cap\sigma=\emptyset$.
                    \item[(e)] For any $\sigma\in\Lambda_{12}(n)$, there is a unique $\Gamma_{12}(\sigma)\in \Lambda_{2+}(n)$ such that $\Gamma_{12}(\sigma)\cap\sigma=\emptyset$.
                \end{itemize}
            \end{lemma}
            \begin{proof}
                \indent Part (a) and (c) uses standard counting arguments and part (b) uses binomial coefficient identities.
                \\
                \indent Part (d) and part (e) are similar proofs, so we will just prove part (d).  For any $\sigma\in\Lambda_2(n)$, it is of the form
                \begin{align*}
                    \sigma=(2,a_2,\dots, a_{n-1})
                \end{align*}
                with $a_2,\dots, a_{n-1}\in\{3,\dots, 2n-1\}$.  The set  $\{3,\dots, 2n-1\}\setminus \{a_2,\dots, a_{n-1}\}$ has cardinality $2n-3 - n+2=n-1$.  Hence, there is only one other element that is disjoint with $\sigma$, i.e. $\Gamma_2(\sigma)$.  
            \end{proof}
            \indent
            \\
            \indent Note that the functions $\Gamma_2$ and $\Gamma_{12}$ are their own inverses, since if we have $\Gamma_2(\sigma)\cap \sigma=\emptyset$, then $\Gamma_2\Gamma_2(\sigma)\cap \Gamma_2(\sigma)=\emptyset$, then by uniquness we must have $\Gamma_2\Gamma_2(\sigma)=\sigma$, and a similar argument for $\Gamma_{12}$.     

            \indent Next, we will see how the new functions $\Gamma_2$ and $\Gamma_{12}$ interact with the signed distance function in the following lemma.

            \begin{lemma}\label{signed dist prop}
                \indent We have the following properties.
                \begin{itemize}
                    \item[(a)] Suppose $\omega,\tau\in \Lambda_{2}(n)$ such that $\omega\leq \tau$, then $\Gamma_{2}(\tau)\leq \Gamma_{2}(\omega)$.
                    \item[(b)] If $\omega,\tau\in\Lambda_{2}(n)$ such that $d_S(\tau,\omega)=1$, then $d_S(\Gamma_2(\omega),\Gamma_2(\tau))=1$.
                    \item[(c)] For any $\omega\in \Lambda_{2}(n)$, we have
                    \begin{align*}
                        d_S((2,\dots, n),\omega) = d_S(\Gamma(\omega),(n+1,\dots, 2n-1)).
                    \end{align*}
                    \item[(d)] If $\omega,\tau\in\Lambda_2(n)$, then $|\omega\cap\tau|=n-2$ if and only if $|\Gamma_2(\omega)\cap \Gamma_2(\tau)|=n-3$.  
                \end{itemize}
            \end{lemma}
            \begin{proof}
                \indent For part (a), if $\omega\leq \tau$, then there exists an $i$ with $1\leq i\leq n-2$ such that
                \begin{align*}
                    \omega = (2,a_2,\dots, a_{i-1},a_i,\dots, a_{n-1})\\
                    \tau = (2,a_2,\dots, a_{i-1}, b_i,\dots, b_{n-1})
                \end{align*}
                with $a_i< b_i$.  Therefore,
                \begin{align*}
                    \Gamma_2(\omega) = (c_1,\dots, c_r, c_{r+1},\dots, c_{n-1})\\
                    \Gamma_2(\tau) = (c_1,\dots, c_{r}, d_{r+1},\dots, d_{n-1}),
                \end{align*}
                where the first $r$ elements are the same since the first $i-1$ elements are the same for $\tau$ and $\omega$ and $d_{r+1}=a_i$ since $a_{i-1}<a_i<b_i$ and $c_{r+1}$ is any element $a_i< c_{r+1}$, and this shows $\Gamma_2(\tau)\leq \Gamma_2(\omega)$.
                \\
                \indent For part (b), suppose for a contradiction that $d_S(\Gamma_2(\omega),\Gamma_2(\tau))\neq 1$.  If $d_S(\Gamma_2(\omega),\Gamma_2(\tau))=0$, then $\Gamma_2(\omega)=\Gamma_2(\tau)$ and by part (a), we have $\omega=\tau$ which contradicts $d_S(\omega,\tau)=1$.  If $d_S(\Gamma_2(\Omega),\Gamma_2(\tau))>1$, then there exists $\sigma\in \Lambda_2(n)$ such that 
                \begin{align*}
                    \Gamma_2(\omega)<\sigma<\Gamma_2(\tau)
                \end{align*}
                which implies 
                \begin{align*}
                    \tau < \Gamma_2(\sigma) < \omega
                \end{align*}
                and this contradicts $d_S(\tau,\Omega)=1$. 
                \\
                \indent For part (c), we will prove this by induction.  First, it is clear
                \begin{align*}
                    d_S((2,\dots, n),(2,\dots, n)) &= d_S((n+1,\dots, 2n-1),(n+1,\dots, 2n-1))\\
                &=d_S((\Gamma_2((2,\dots, n)), (n+1,\dots, 2n-1)).
                \end{align*}
                Next, suppose $\sigma\in\Lambda_2(n)$ such that $(2,\dots, n)\leq \sigma$ and it is not the last element in $\Lambda_2(n)$ such that
                \begin{align*}
                    d_S((2,\dots, n),\sigma) = d_S(\Gamma_2(\sigma), (n+1,\dots, 2n-1)).
                \end{align*}
                If $\tau\in \Lambda_2(n)$ such that $d_S(\sigma,\tau)=1$, then we have
                \begin{align*}
                    d_S((2,\dots, n),\tau) &= d_S((2,\dots, n),\sigma) + d_S(\sigma,\tau)\\
                    &=d_S(\Gamma_2(\sigma), (n+1,\dots, 2n-1)) + d_S(\Gamma_2(\tau),\Gamma_2(\sigma))\\
                    &=d_S(\Gamma_2(\tau), (n+1,\dots, 2n-1)),
                \end{align*}
                which proves part (c). 
                \\
                \indent For part (d), if $|\omega\cap\tau|=n-2$, then we have
                \begin{align*}
                    \omega &= (2,a_1,\dots, a_i, a_{i+1},\dots, a_{n-2})\\
                    \tau &= (2, a_1,\dots, a_i, b_{i+1},a_{i+2},\dots, a_{n-2})
                \end{align*}
                where $a_{i+1}\neq b_{i+1}$.  We have
                \begin{align*}
                    \{3,\dots, 2n-1\}\setminus \{a_1,\dots, a_{n-2}\} \cap \{3,\dots, 2n-1\}\setminus \{a_1,\dots, a_i,b_{i+1},a_{i+2},\dots, a_{n-2}\}
                \end{align*}
                has $n-3$ elements since we took out almost all the same elements except for one in both that are not common. This implies $|\Gamma_2(\omega)\cap \Gamma_2(\tau)|=n-3$.
            \end{proof}
            \indent Similar statements hold for $\Gamma_{12}$ with similar proofs.

        \subsection{The Adjacency Matrix of $\cO_n$}
        \indent With the technical information about the vertices, we can apply this to show $B(\cO_n)$ has a very useful block matrix form that is conducive to an easy inductive argument.
        \begin{lemma}\label{Description of B^2}
                \indent For any $n\geq 2$, we have 
                \begin{align*}
                    B(O_n)^2_{\tau,\omega} = \begin{cases}
                        n & \text{if $\tau=\omega$}\\
                        1 & \text{if $|\tau\cap\omega|=n-2$}\\
                        0 & \text{if $0\leq |\tau\cap\omega|\leq n-3$}
                    \end{cases}.
                \end{align*}
            \end{lemma}
            \begin{proof}
                \indent By definition, for any $\tau,\omega\in\Lambda(n)$ we have
                \begin{align*}
                    (B(O_n)^2)_{\tau,\omega}& = \sum_{\sigma\in\Lambda(n)} B(O_n)_{\tau,\sigma}B(O_n)_{\sigma,\omega}\\
                    &=\sum_{\substack{\sigma\in\Lambda(n)\\ \sigma\cap\tau=\sigma\cap\omega=\emptyset}}1.
                \end{align*}
                \indent There are only three cases we need to consider to describe our matrix:
                \begin{itemize}
                \setlength{\itemindent}{1cm}
                    \item $\tau=\omega$,\\
                    \item $|\tau\cap\omega|=n-2$,\\
                    \item  and $0\leq |\tau\cap\omega|\leq n-3$.
                \end{itemize}  
                \indent For $\tau=\omega$, by lemma \ref{Prop of Lambda(n)}, the number of $\sigma\in \Lambda(n)$ such that $\sigma\cap \tau=\emptyset$ is $n$.  Therefore,
                \begin{align*}
                    (B(O_n)^2)_{\tau,\tau}=n.
                \end{align*}
                \\
                \indent When $|\tau\cap \omega|=n-2$, then the number of elements in $\{1,\dots, 2n-1\}\setminus (\tau\cup \omega)$ is $n-1$.  Hence, there are only $\binom{n-1}{n-1}=1$ element that intersect both trivially.  Therefore,
                \begin{align*}
                     (B(O_n)^2)_{\tau,\omega} = 1
                \end{align*}
                in this case.
                \\
                \indent Finally, when $0\leq |\tau\cap\omega|\leq n-3$, we have
                \begin{align*}
                   | \{1,\dots, 2n-1\}\setminus (\tau\cup \omega)|\leq n-2.
                \end{align*}
                Therefore, there is no other element that intersects both trivially and hence
                \begin{align*}
                     (B(O_n)^2)_{\tau,\omega}=0.
                \end{align*}
                \\
                \indent This gives us our result. 
            \end{proof}
            
            \indent Next, we will describe the adjacency matrix for $O_n$, using the fact that it will contain the adjacency matrix of $O_{n-1}$ in a vary particular way. 
            \\
            \begin{lemma} \label{adjacency form}
            We have the following properties for our adjacency matrices.
        \begin{itemize}
            \item[(a)] Let $N_n=|\Lambda_1(n)|$ and $M_n=|\Lambda_{1+}(n)|$.  The adjancency matrix $B(O_n)$ satisfies the following block form
            \begin{align*}
                B(O_n) = \begin{pmatrix}
                    0 & C(O_n)\\
                    C(O_n)^T & J
                \end{pmatrix}
            \end{align*}
            where $J$ is the anti-diagonal identity-matrix of the correct size, and $C(O_n)$ is $N_n\times M_n$- submatrix for $\omega\in\Lambda_1(n)$ and $\tau\in\Lambda_+(n)$ such that 
            \begin{align*}
               C(O_n)_{\omega,\tau}=B(O_n)_{\omega,\tau}. 
            \end{align*}
            \item[(b)] The matrix $C(O_n)$ satisfies the block form
            \begin{align*}
                C(O_n) = \begin{pmatrix}
                    0 & D(O_n)\\
                    B(O_{n-1}) & J
                \end{pmatrix}
            \end{align*}
            for some matrix $D(O_n)$ with $D(O_n)^TD(O_n)+I = JB(O_{n-1})^2J$, where $J$ is the anti-diagonal matrix of the correct size.  
        \end{itemize}
    \end{lemma}
    \begin{proof}
        \indent For part (a), the top left block has rows and columns indexed by $\Lambda_1(n)$.  Since every element of $\Lambda_1(n)$ has a $1$ in it, then they all intersect non-trivially, hence this block is the zero matrix. The bottom right block has rows and columns indexed by $\Lambda_{1+}=\Lambda_{2}(n)\coprod \Lambda_{2+}(n)$.  By lemma \ref{Prop of Lambda(n)}, for each $\sigma\in\Lambda_2(n)$, there is a unique $\Gamma_2(\sigma)\in\Lambda_{2+}(n)$ such that $\Gamma_2(\sigma)\cap\sigma=\emptyset$ with 
        \begin{align*}
            d_S((2,\dots, n),\sigma)=d_S(\Gamma_2(\sigma),(n+1,\dots, 2n-1)).
        \end{align*}
        This shows that the bottom right block is $J$ from the ordering.  
     The top right and bottom left blocks comes from $B(O_n)$ is a symmetric matrix.  
        \\
        \indent For part (b), the top left block is zero since the rows are indexed by $\Lambda_{12}(n)$ and the columns are indexed by $\Lambda_2(n)$ which all have $2$'s in them.  The bottom right block is $J$ as as a consequence of lemma \ref{signed dist prop}.
         For the bottom left block of the matrix $C(O_n)$, the rows are indexed by $\Lambda_{12+}(n)$ and the columns are indexed by $\Lambda_2(n)$.  We have well-defined bijections $\varphi_n:\Lambda(n-1)\rightarrow \Lambda_{2}(n)$ and $\psi_n:\Lambda(n-1)\rightarrow \Lambda_{12+}(n)$ defined as
        \begin{align*}
            \varphi_n((a_1,\dots, a_{n-2})) = (2,a_1+2,\dots, a_{n-2}+2)\\
            \psi_n((a_1,\dots, a_{n-2})) = (1, a_1+2,\dots, a_{n-2}+2).
        \end{align*}
        \\
        \indent It is clear that the intersection is preserved: if $\sigma,\tau\in\Lambda(n-1)$, then $\sigma\cap\tau=\emptyset$ if and only if $\varphi_{n}(\sigma)\cap\psi_n(\tau)=\emptyset$.  With these bijections, we have that the bottom left block is exactly $B(O_{n-1})$. For the top right block, its rows are indexed by $\Lambda_{12}(n)$ and its columns are indexed by $\Lambda_{2+}(n)$, and denote this submatrix as $D(O_n)$.  By definition, for any $\tau,\omega\in\Lambda_{2+}(n)$ we have
        \begin{align*}
            (D(O_n)^TD(O_n))_{\tau,\omega} &= \sum_{\sigma\in\Lambda_{12}(n)} D(O_n)_{\sigma,\tau} D(O_n)_{\sigma,\omega}\\
            &= \sum_{\substack{\sigma\in\Lambda_{12}(n)\\ \sigma\cap\tau=\sigma\cap\omega=\emptyset}} 1.
        \end{align*}
        Following the same proof as in lemma \ref{Description of B^2}, we obtain
        \begin{align*}
            (D(O_n)^TD(O_n))_{\tau,\omega} = \begin{cases}
                n - 2 & \text{if $\tau=\omega$}\\
                1 & \text{if $|\tau\cap\omega|=n-2$}\\
                0 & \text{if $0\leq |\tau\cap\omega|\leq n-3$}
            \end{cases}
        \end{align*}
        \\
        \indent Next, we will show $D(O_n)$ has the required property $D(O_n)^TD(O_n) + I = J B(O_{n-1})^2J$. The maps $\varphi_n$ and $\Gamma_2$ induce the following linear map
        $$\begin{tikzcd}
            k\Lambda(n-1)\arrow[r, "\varphi_n"] & k\Lambda_2(n)\arrow[r, "\Gamma_2"] & k\Lambda_{2+}(n)\\
        \end{tikzcd}$$
        which has matrix representation $J$ by lemma \ref{signed dist prop}.  By lemma \ref{Description of B^2} and our description for $D(O_n)^TD(O_n)$ we have the following: if $\tau = (a_1,\dots, a_{n-1})\in k\Lambda_{2+}(n)$, then 
        \begin{align*}
            JB(\cO_{n-1})^2J^{-1}(\tau) & =JB(\cO_{n-1})^2(\varphi_n^{-1}\Gamma_2(\tau))\\
            &= J((n-1)\varphi_n^{-1}\Gamma_2(\tau) + \sum_{|\varphi_n^{-1}\Gamma_2(\tau)\cap \omega|=n-3}\omega)\\
            &= (n-1)\tau + \sum_{|\varphi_n^{-1}\Gamma_2(\tau)\cap \omega|=n-3}\Gamma_2(\varphi_n(\omega))\\
            &= (n-1)\tau + \sum_{|\tau\cap \Gamma_2(\varphi_n\omega))|=n-2}\Gamma_2(\varphi_n(\omega))\\
            & = (n-1)\tau + \sum_{|\tau\cap \sigma|=n-2}\sigma
        \end{align*}
        which is exactly $D(\cO_n)^TD(\cO_n) +I$.

        , we see that $JB(O_{n-1})^2J^{-1} = JB(O_{n-1})^2J = D(O_n)^TD(O_n) + I$, which completes the proof.   
    \end{proof}

        \indent This shows that $\{(B(O_n),C(O_n),D(O_n))\}$ is a balanced matrix sequence as in the definition \ref{balancedmatrix} in the appendix, which gives us a way to relate the characteristic polynomial of $B(O_n)$ with $B(O_{n-1})$ by lemma \ref{Balanced Theorem}. 

    \section{Proof of Theorem \ref{Spec of O_n}}\label{proof of spec}

    For this subsection, we give the proof for theorem \ref{Spec of O_n}, which is a crucial part in showing $R_{n,d}^{\perp}$ has the correct dimension in characteristic $0$ to prove our main result.  Since $\{(B(\cO_n),C(\cO_n),D(\cO_n)\}$ is a balanced matrix sequence, we have the sequence $(N_n,M_n)$ as in definition \ref{balancedmatrix} where
    \begin{align*}
        N_n = \begin{pmatrix} 2n-2 \\ n-2\end{pmatrix}\\
        M_n = \begin{pmatrix} 2n-2\\ n-1\end{pmatrix}\\
    \end{align*}
    by lemmas \ref{Prop of Lambda(n)} and \ref{adjacency form}.  By definition of the Catalan numbers, we have
    \begin{align*}
        M_n - N_n = C_{n-1}.
    \end{align*}
        \begin{proof}[proof of theorem \ref{Spec of O_n} ]
            \indent We will prove this by induction on $n$.  For $n=2$, we already described the spectrum of $O_2$, which gives us $m_{2,1} = 2=\cE_{1,0}$ and $m_{2,2} = 1 = \cE_{1,1}$.  
            \\
            \indent Next, suppose $n-1>2$ and $O_{n-1}$ satisfies the properties in the statement of the theorem.  Since $\{(B(O_n),C(O_n),D(O_n))\}$ is a balanced matrix sequence, hence by lemma \ref{Balanced Theorem}, we have
            \begin{align*}
                P_{B(O_n)}(x) = -(-1)^{N_{n} - N_{n-1}} x^{N_n-M_n} P_{B(O_{n-1})}(-x)^2 P_{B(O_{n-1})}(x-1) P_{B(O_{n-1})}(x+1)
            \end{align*}
            where $P_A$ is the characteristic polynomial of a matrix $A$, and $N_{n} - M_n = -C_{n-1}$.  Since we know the eigenvalues and multiplicities of $B(O_{n-1})$, we have $P_{B(O_n)}(x)$ is of the form
            \begin{align*}
               \pm x^{-C_n} &(x+\lambda_1^{n-1})^{2m_{n-1,1}}\cdots (x+\lambda_{n-1}^{n-1})^{2m_{n-1,n-1}} \\
               &\cdot(x-(\lambda_1^{n-1}+1))^{m_{n-1,1}}\cdots (x-(\lambda_{n-1}^{n-1}+1))^{m_{n-1,n-1}}\\
               &\cdot(x-(\lambda_1^{n-1}-1))^{m_{n-1,1}}\cdots (x-(\lambda_{n-1}^{n-1}-1))^{m_{n-1,n-1}}
            \end{align*}
            \indent In this product, we either have $(x-(\lambda_1^{n-1} - 1)^{m_{n-1,1}}=x^{m_{n-1,1}}$ or $(x-(\lambda_1^{n-1} + 1)^{m_{n-1,1}}=x^{m_{n-1,1}}$ depending on if $n$ is odd or even.  In either case, since $m_{n-1,1}= \cE_{n-2,0}=C_{n-1}$, this cancels out the $x^{-C_{n-1}}$ in the front of the product.  Furthermore, the eigenvalues switch signs, and we obtain $\lambda_n^n=n$ and $\lambda_i^n=-\lambda_i^{n-1}$ for $1\leq i\leq n-1$.  
            \\
            \indent From this product, the multiplicities are
            \begin{align*}
                m_{n,i} &= m_{n-1,i-1} + 2m_{n-1,i} + m_{n-1,i+1} = \\
                &=\cE_{n-2,i-2} + 2\cE_{n-2,i-1} + \cE_{n-2,i}\\
                &=\cE_{n-1,i-1}
            \end{align*}
            for $1\leq i\leq n$ where $m_{i,j}=0$ for $j>i$ and $j<0$.  This proves the result by induction.  
        \end{proof}

\bibliography{sn-bibliography}

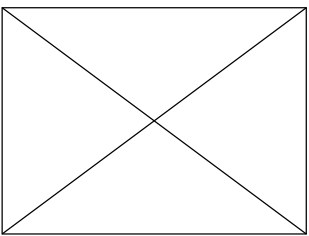
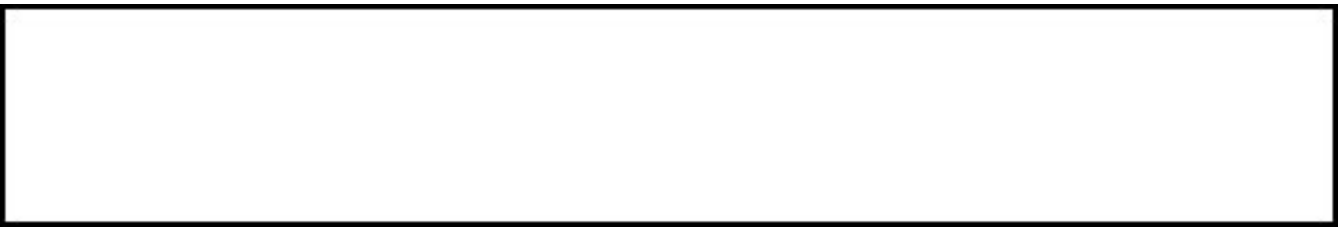
\end{document}